\documentclass[12pt]{amsart}
\usepackage{amsfonts, latexsym, amssymb, amsgen, amsbsy, amstext, amsopn, amsmath, 
txfonts, mathrsfs, yfonts, dsfont}
\usepackage[bf]{caption}     
\usepackage[dvips]{graphicx}
\usepackage{psfrag}          

\renewcommand{\H}{\mathbb{H}}

\newcommand{\E}{\mathbb{E}}
\newcommand{\G}{\mathbb{G}}

\newcommand{\N}{\mathbb{N}}

\newcommand{\R}{\mathbb{R}}
\renewcommand{\S}{\mathbb{S}}

\newcommand{\vE}{\varmathbb{E}}

\newcommand{\vG}{\varmathbb{G}}
\newcommand{\vH}{\varmathbb{H}}

\newcommand{\vM}{\varmathbb{M}}

\newcommand{\vX}{\varmathbb{X}}

\newcommand{\vT}{\varmathbb{T}}

\newcommand{\cD}{\mathcal{D}}

\newcommand{\cF}{\mathcal{F}}

\newcommand{\cH}{\mathcal{H}}

\newcommand{\cL}{\mathcal{L}}

\newcommand{\cJ}{\mathcal{J}}

\newcommand{\diam}{\mbox{diam}}

\newcommand{\Lip}{\mbox{Lip}}

\newcommand{\dist}{\mbox{dist}}

\newcommand{\ep}{\varepsilon}

\newcommand{\sm}{\setminus}

\newcommand{\lan}{\langle}
\newcommand{\ran}{\rangle}

\newcommand{\lra}{\longrightarrow}

\newcommand{\der}{\partial}

\newtheorem{The}{Theorem}[section]
\newtheorem{Lem}{Lemma}[section]
\newtheorem{Rem}{Remark}[section]
\newtheorem{Def}{Definition}[section]
\newtheorem{Pro}{Proposition}[section]
\newtheorem{Cor}{Corollary}[section]

\begin{document}

\title
[Radon-Nikodym property and area formula]
{Radon-Nikodym property and area formula for Banach homogeneous group targets}
\author{Valentino Magnani}
\address{Valentino Magnani, Dipartimento di Matematica \\
Largo Bruno Pontecorvo 5 \\ 56127, Pisa, Italy}
\email{magnani@dm.unipi.it}
\author{Tapio Rajala}
\address{Tapio Rajala, Scuola Normale Superiore\\
Piazza dei Cavalieri 7\\
56127, Pisa, Italy}
\email{tapio.rajala@sns.it}
\thanks{The authors acknowledge the support of the European Project ERC AdG *GeMeThNES*.}
\thanks{The second author also acknowledges the support of the Academy of Finland, project no. 137528.}
\date{\today}

\begin{abstract}
We prove a Rademacher-type theorem for Lipschitz mappings from a subset of a Carnot group to
a Banach homogeneous group, equipped with a suitably weakened Radon-Nikodym property.
We provide a metric area formula that applies to these mappings and more generally 
to all almost everywhere metrically differentiable Lipschitz mappings defined on a Carnot group.
\end{abstract}
\maketitle

\tableofcontents

\newpage 

%
%
%
%

\section{Introduction}

A Banach space that has the so-called Radon-Nikodym property, in short RNP, satisfies the condition that all Lipschitz curves taking values in this space are almost everywhere 
Fr\'echet differentiable.
This property is equivalent to the fact that for a Banach space $\vX$
all $\vX$-valued Lipschitz mappings defined on the Euclidean space are 
almost everywhere Fr\'echet differentiable if and only if $\vX$ has the RNP.

Recently, J. Cheeger and B. Kleiner have shown that Banach spaces with the RNP form
 the sharp class of targets for which a Rademacher-type theorem holds, when the source space 
is a doubling metric space, satisfying the Poincar\'e inequality, in short a PI space, \cite{CheKle09}. 
On the other hand, this result does not cover the Pansu differentiability of Lipschitz mappings, 
\cite{Pan89}, that is a Rademacher-type theorem between Carnot groups. 
In fact, although here the target is finite dimensional, the difference arises from the intrinsic nonlinear nature of the notion of Pansu differentiability.

The first result in this work is an infinite dimensional version of almost everywhere 
Pansu differentiability of Lipschitz mappings, \cite{Pan89}. In other words, 
we show that replacing PI spaces with Carnot groups allows us to extend the family of targets for which we have a Rademacher-type theorem. 
These targets are Banach Lie groups metrized by a suitable left invariant distance.
In the commutative case, they also include the classical Banach spaces, but their main feature is that the RNP is required only on a special closed subspace, that is the so-called {\em horizontal subspace}. Notice that in our Banach Lie groups this subspace is in general an infinite dimensional Banach space.

These Lie groups can be naturally called {\em Banach homogeneous groups}, since they are a straightforward generalization of their well known finite dimensional version. For the finite dimensional case, the reader can consult for instance \cite{FS82}.
A simple way to present these infinite dimensional versions may consist in requiring 
the validity of the crucial properties that hold in the finite dimensional case, as the existence 
of a group operation with special structure, the existence of a homogeneous norm,
along with dilations and so on. 
This presentation by axioms can be found in \cite{Rog07}. 

We follow a different approach, detecting these groups as {\em graded nilpotent Banach Lie algebras},
since all the above mentioned properties are just consequences, see Section~\ref{Sect:infdim}.
In fact, one can see a Banach homogeneous group as a Banach space $\vM$ equipped with a graded nilpotent Lie product that turns it into a Banach Lie algebra. Thus, we automatically get the group operation by the Dynkin formula for the BCH series, that locally converges in general Banach Lie algebras, \cite{Dyn53}. Of course, in our case, this series is just a finite sum, that clearly everywhere converges, since we consider a nilpotent Lie algebra.

In sum, we equip $\vM$ with three structures, since it is a Banach space,
a Banach Lie algebra and also Banach Lie group. Its main feature is the decomposition into the
direct sum $\vM=H_1\oplus\cdots\oplus H_\iota$, where $H_j$ are closed subspaces of $\vM$, seen as Banach spaces. 
This yields a precise gradation on the Lie algebra structure of $\vM$ that allows 
us to introduce dilations $\delta_r:\vM\lra\vM$, that are automatically group isomorphisms.
Furthermore, one can also construct a homogeneous norm on $\vM$ that respects both the group operation and dilations,
hence defining the metric structure of $\vM$, see Section~\ref{Sect:infdim} for more details.

As in the finite dimensional case, one can also define the special class of Banach stratified groups, or Carnot groups, since the additional condition that $H_1$ Lie generates $\vM$ in the finite dimensional case \cite{FS82,Pan89}, 
can be also stated in the infinite dimensional case. This was already pointed out in the seminal work by M. Gromov, \cite{Gr1}.
However, we will focus our attention on the larger class of Banach homogeneous group targets, 
that presents some additional difficulties in the proof of the almost everywhere differentiability of Lipschitz mappings, as explained below. 

Several examples of infinite dimensional Banach homogeneous groups will be presented in Section~\ref{Sect:infdim}. 
We mainly exploit a natural product construction by means of the Banach spaces $\ell^p$ of
$p$-summable sequences. 
The simplest example of infinite dimensional Banach homogeneous group is the well known 
Heisenberg group modeled on $H^2\times\R$, where $H$ is a real Hilbert space
of scalar product $\lan\cdot,\cdot\ran$. For any $(h_1,h_2,t_1),(h_1',h_2',t_2')\in H^2\times\R$,
the group operation is defined as follows
\begin{equation}\label{HeisQuant}
(h_1,h_2,t_1)(h_1',h_2',t_2')=(h_1+h_1',h_2+h_2',t_2+t_2'+\lan h_1,h_2'\ran-\lan h_2,h_1'\ran).
\end{equation}
This product arises from the quantization relations of the Heisenberg algebra realized in  
Quantum Mechanics, see for instance Chapter XII, Section 3 of \cite{Stein93}.
Notice that this group has an underlying Hilbert space structure. 
In Subsection~\ref{infHeis}, we introduce the infinite product of Heisenberg groups $\vH^\infty$, whose
underlying Banach structure is given by $(\ell^2)^2\times \ell^1$. In Subsection~\ref{infCarnot} we 
present a construction to obtain an infinite product of the same Carnot group. This provides many 
Banach homogeneous groups whose underlying linear space is a genuinely infinite dimensional Banach space
and we will also see that we have some freedom in the choice of the Banach topology.
It is clear that one could use a similar construction also for products of different Carnot groups.
Motivated by the simple case given by \eqref{HeisQuant} that arises 
from the Heisenberg group of Quantum Mechanics, one might also expect further physical
interpretations for special classes of Banach homogeneous groups.

We wish to clarify that the terminology ``Carnot group'' will refer throughout to a finite
dimensional group. The notion of differentiability in Carnot groups has been introduced by P. Pansu. 
In his celebrated 1989 work about the rigidity of both hyperbolic quaternionic spaces and of the Cayley hyperbolic plane, 
he shows that a Lipschitz mapping from an open set of a Carnot group to another Carnot group is almost everywhere
differentiable, see \cite{Pan89}. In Definition~\ref{def:hdiff} we recall this notion of differentiability in our framework.
\begin{The}\label{thm:Pansudiff}
Let $\vM$ be a Banach homogeneous group such that its subspace $H_1$ has the RNP.
If $\G$ is any Carnot group and $A \subset \G$, then any Lipschitz mapping 
$f \colon A \to \vM$ is almost everywhere differentiable.
\end{The} 
If the target is either another Carnot group or a finite dimensional graded Lie group, then
this theorem yields the known results on almost everywhere differentiability of Lipschitz mappings.
In fact, all finite dimensional linear spaces have the RNP.
The proof Theorem~\ref{thm:Pansudiff} follows an approach that differs from all the other approaches known to us. 
We only rely on a suitable application of the Dynkin formula that gives an explicit expression for
the addends appearing in the finite expansion of the group operation, \cite{Dyn47}. 
Further technical difficulties arise from the fact that we 
consider a general subset of the Carnot group $\G$. 
In Theorem~\ref{the:Hdiff} we establish the almost everywhere differentiability 
of Lipschitz curves taking values in $\vM$, equipped with the metric structure of homogeneous group,
and being defined on an arbitrary subset of the real line.
This is a nontrivial fact, since Lipschitz extensions of our curve, with respect to the
homogeneous distance $\rho$ on $\vM$, may not exist on an open interval.
In fact, one can easily find finite dimensional homogeneous groups that are not connected by rectifiable curves.
We overcome this difficulty by a suitable use of Lipschitz extensions for Banach-valued Lipschitz curves. This result is due to  W. B. Johnson, J. Lindenstrauss and G. Schechtman, \cite{JLS86}.

Another issue related to Theorem~\ref{thm:Pansudiff}, 
may concern the actual existence of nontrivial Lipschitz mappings. 
We wish to make sure that there are Lipschitz mappings that are not a mere composition 
of a Carnot group-valued Lipschitz mapping with a Lipschitz embedding into a Banach homogeneous group. 
In Subsection~\ref{infLip}, we construct Lipschitz mappings that cannot have the form previously described. In fact, we consider a suitable infinite product of a family of Lipschitz mappings 
$\{f^k\}_{k\geq0}$, under the condition that all vanish at some point. The corresponding product mapping $G$ turns out to be a Lipschitz mapping taking values in $\vH^\infty$. 
Of course, in the case all mappings $f^k$ do not vanish at some point,
the corresponding product mapping $G$ is an example of Lipschitz mapping with {\em infinite dimensional image}.
Since the horizontal layer of $\vH^\infty$ has the RNP, our Theorem~\ref{thm:Pansudiff} shows that
$G$ is also almost everywhere differentiable, when the source space is any Carnot group.
It is worth to mention that all examples of Banach homogeneous groups given in Section~\ref{Sect:infdim}
have the RNP, so in particular their horizontal subspace has also the RNP.

The second part of this work is devoted to an area formula for a rather general class of metric space-valued Lipschitz mappings, that also includes those of Theorem~\ref{thm:Pansudiff}.
For a general metric space target, the choice of the source space is crucial.
In fact, for metric space-valued Lipschitz mappings defined on a subset of a Euclidean space,
B. Kirchheim has established their almost everywhere metric differentiability and also the
corresponding area formula, \cite{Kir94}. 
It is important to remark that Carnot groups have a sufficiently rich structure to introduce the 
notion of metric differentiability, when any of their subsets constitutes the source space, see Definition~\ref{def:metdif}. In this case, the metric differential 
is given by a {\em homogeneous seminorm}, namely a continuous function $s:\G\lra[0,+\infty[$ such that $s(x)=s(x^{-1})$, $s(\delta_rx)=rs(x)$ and $s(x\cdot y)\leq s(x)+s(y)$ for all $x,y\in\G$ and $r>0$.
The additional condition that $s(x)=0$ implies $x=0$ means that $s$ is a {\em homogeneous norm}.

On one side, when the source space is a noncommutative Carnot group, such as the Heisenberg group, 
then counterexamples to the metric differentiability of Lipschitz mappings can be constructed, \cite{KirMag2}. 
On the other side, if we restrict metric differentiability to horizontal directions, then
we still have an almost everywhere (horizontal) metric differentiation for metric space-valued 
Lipschitz mappings on Carnot groups, \cite{Pau01}.

Theorem~\ref{thm:Pansudiff} clearly provides new cases where metric space-valued Lipschitz mappings on Carnot groups are almost everywhere metrically differentiable. Other novel targets where the almost everywhere  metric differentiability holds can be found by another recent result of J. Cheeger and B. Kleiner, \cite{CheKle10b}. 
In fact, one can notice that the seminorm $\|\cdot\|_x$ of Theorem~1.3 in \cite{CheKle10b} can be seen as a homogeneous seminorm on the whole Heisenberg group $\H$, therefore the limit in the statement of this theorem, with $z_2$ equal to the unit element, exactly yields the almost everywhere metric differentiability of Lipschitz mappings from $\H$ to $L^1(0,1)$, see also \cite{CheKle10a}.
We observe that in all previously mentioned cases, where the almost everywhere metric differentiability holds, one can apply the following new metric area formula.
\begin{The}\label{gmetarea}
Let $A\subset\G$ be measurable, let $f:A\lra Y$ be Lipschitz and
almost everywhere metrically differentiable. It follows that 
\begin{equation}\label{gareametric}
\int_A J(mdf(x)) \, d\cH^Q_d(x)=\int_Y N(f,y)\, d\cH^Q_\rho(y)\,,
\end{equation}
where $N(f,y)=\sharp\big(f^{-1}(y)\big)$ for all $y\in Y$ is the multiplicity function,
$d$ is the homogeneous distance of $\G$,
$\rho$ is the metric of $Y$ and $Q$ is the Hausdorff dimension of $\G$.
\end{The}
As usual, the point of an area formula is its notion of Jacobian.
The {\em metric Jacobian} $J(s)$ of the homogeneous seminorm $s$ is defined as follows
\begin{eqnarray}\label{metJac}
 J(s)=\left\{\begin{array}{ll} \displaystyle \frac{\cH_s^Q(B_1)}{\cH_d^Q(B_1)} & \mbox{if $s$ is a homogeneous norm} \\
 0 & \mbox{otherwise}  \end{array}\right.\,.
\end{eqnarray}
If $\G$ is a Euclidean space, then \eqref{metJac} yields the Jacobian of \cite{Kir94}. If the target is a Banach homogeneous group $\vM$ equipped with a distance $\rho$ given by a homogeneous norm, then
we have to observe that differentiability with differential $L:\G\lra\vM$ implies metric differentiability with homogeneous seminorm $h\to s_L(h)=\rho\big(L(h),0\big)$ with $h\in\G$.
Thus, we get a more explicit formula for \eqref{metJac}, that in the special case 
$\vM$ is another Carnot group fits into the sub-Riemannian Jacobian introduced in \cite{Mag}, see Remark~\ref{rem:BanHomJ} for more comments.

Concerning the proof of \eqref{gareametric}, a substantial difference in our approach with respect to 
that of \cite{Kir94} is in the proof of the negligibility of the image of points where the metric differential is not a homogeneous norm. In \cite{Kir94}, this fact is achieved combining the integral representation of Kirchheim's Jacobian with the use of $\ep$-approximating graph extensions of the mapping, as in \cite{Fed}.
It is somehow surprising that our argument is more elementary, since it only uses the definition of metric differentiability, according to Lemma~\ref{LipNegl}, 
without any use of the notion of metric Jacobian.

We should also mention that the area formula in Carnot groups leads to an algebraic characterization of {\em purely $\G$-unrectifiable} stratified groups, \cite{Mag10}. A purely $\G$-unrectifiable metric space $(Y,\rho)$ has the property that the image of any $Y$-valued Lipschitz mapping from a subset of the $Q$-dimensional Carnot group $\G$ has vanishing Hausdorff measure $\cH^Q_\rho$.

As a byproduct of \eqref{gareametric}, for any fixed Carnot group $\G$, a Banach homogeneous group such that its horizontal subspace has the RNP is purely $\G$-unrectifiable if none of its homogeneous subgroups is isomorphic to $\G$. The trivial case is that of a Banach space with the RNP, that is clearly purely $\G$-unrectifiable whenever $\G$ is noncommutative. For instance, if $\G$ has step higher than two, then any two step Banach homogeneous group whose horizontal subspace has the RNP must be purely $\G$-unrectifiable. Clearly, other analogous cases could be conceived. 
We wish to clarify that in the preceding discussion, we have referred to the notion of 
isomorphism using the h-homomorphisms introduced in Definition~\ref{hhom}.

Finally, we remark that any metric space $Y$ that is purely $\G$-unrectifiable has in particular the property that the group $\G$ cannot admit any bi-Lipschitz embedding into $Y$. This clearly
provides new bi-Lipschitz non-embeddability theorems with infinite dimensional target.

%
%
%
%

\section{Banach homogeneous groups}\label{Sect:infdim}

We start from the notion of {\em Banach Lie algebra}, namely a Banach space $\vM$ equipped with a continuous, bilinear and skew-symmetric mapping $[\cdot,\cdot]:\vM\times \vM\lra\vM$ that satisfies the Jacobi identity.
A {\em nilpotent Banach Lie algebra} $\vM$ is characterized by the existence of a positive integer $\nu\in\N$ 
such that whenever $x_1,x_2,\ldots,x_{\nu+1}\in\vM$, we have
\[
[\cdots[[x_1,x_2],x_3]\cdots], x_\nu],x_{\nu+1}]=0\quad
\]
and there exist $y_1,y_2,\ldots,y_\nu\in\vM$ such that 
\[
[\cdots[[y_1,y_2],y_3]\cdots],y_\nu]\neq0.
\]
The integer $\nu$ is uniquely defined and it gives the {\em step of nilpotence} of $\vM$. 
Therefore the algebra $\vM$ can be equipped with a canonical Banach Lie group operation 
\begin{equation}\label{gopeB}
xy=x+y+\sum_{m=2}^\nu P_m(x,y)\,,
\end{equation}
that is the ``truncated'' Baker-Campbell-Hausdorff series. 
For any $m\geq 2$, the polynomial $P_m$ is given by the {\em Dynkin's formula}
\begin{equation}\label{dynkin}
P_m(x,y)= \mbox{\scriptsize $\displaystyle \sum \frac{(-1)^{k-1}}{k}\frac{m^{-1}}{p_1!q_1!\cdots p_k!q_k!}$}  \; 
\mbox{\small $\underbrace{x\circ\cdots\circ x}_{p_1\; times}\circ\overbrace{y\circ\cdots\circ y}^{q_1\; times}\circ \cdots\circ  \underbrace{x\circ\cdots\circ x}_{p_k\; times}\circ\overbrace{y\circ\cdots\circ y}^{q_k\; times }$ }\,,
\end{equation}
where $x_{i_1}\circ x_{i_2}\circ\cdots \circ x_{i_k}=[\cdots[[x_{i_1},x_{i_2}],x_{i_3}]\cdots],x_{i_k}]$ 
and the sum is taken over the $2k$-tuples $(p_1,q_1,p_2,q_2,\ldots,p_k,q_k)$ such that
$p_i+q_i\geq 1$ for all positive $i,k\in\N$ and $\sum_{i=1}^kp_i+q_i=m$. Notice that 
$P_2(x,y)=[x,y]/2$. 
Formula \eqref{dynkin} was established by E. B. Dynkin in \cite{Dyn47}. 
We say that $\vM$ equipped with the group operation \eqref{gopeB} 
is a {\em Banach nilpotent Lie group}. If we denote by $L(\vM)$ the Lie algebra of $\vM$
as a Lie group, we may wonder whether $L(\vM)$ is isomorphic to $\vM$ seen as a Lie algebra
equipped with the initialliy given Lie product $[\cdot,\cdot]$.
The answer to this question is yes, according to the following proposition, whose proof can be established by
the use of the BCH series for the group expansion.
\begin{Pro}
If $\vM$ is a Banach nilpotent Lie group, then the given Lie algebra structure on $\vM$
is isomorphic to $L(\vM)$.
\end{Pro}
If $S_1,S_2,\ldots,S_n\subset\vX$ are closed subspaces of a Banach space $\vX$ such that the mapping $J:S_1\times\cdots\times S_n\lra\vX$ with $J(s_1,\ldots,s_n)=\sum_{l=1}^n s_l$ is an isomorphism of Banach spaces, then $\vX=S_1\oplus\cdots\oplus S_n$ denotes the corresponding direct sum. 
Any canonical projection on $S_j$ is denoted by $\pi_j:\vX\lra S_j$. 
\begin{Def}\rm 
We say that the Banach space $\vM$ is a {\em Banach homogeneous group} if it is equipped with a 
Banach Lie product $[\cdot,\cdot]:\vM\times\vM\lra\vM$ and there exist $\iota$ closed subspaces 
$H_1,\ldots,H_\iota$ such that $\vM=H_1\oplus\cdots\oplus H_\iota$ and 
whenever $x\in H_i$ and $y\in H_j$ we have $[x,y]\in H_{i+j}$ if $i+j\leq\iota$
and $[x,y]=0$ otherwise. This equips $\vM$ with a special family of Banach isomorphisms
$\delta_r:\vM\lra\vM$, $r>0$, defined by $\delta_rx=r^ix$ if $x\in H_i$ for all $i=1,\ldots,\iota$. 
These mappings are both group and algebra automorphisms of $\vM$ and are called {\em dilations}.
\end{Def}
\begin{Rem}\rm
A Banach homogeneous group can be seen as a {\em Banach graded nilpotent Lie group} equipped with dilations.
This is the natural terminology from the finite dimensional case of graded Lie groups, see \cite{FS82,Good76}.
The decomposition $\vM=H_1\oplus\cdots\oplus H_\iota$ with the properties stated in the previous definition
defines a {\em gradation} of $\vM$.
\end{Rem}
The gradation of $\vM$ along with the Dynkin formula \eqref{dynkin} yields some positive constants $\sigma_1,\ldots,\sigma_\iota$, depending on 
the norm of the Lie product, such that
\begin{equation}\label{Banhom}
\|x\,\|=\max\{\sigma_i|x_i|^{1/i}:\,1\leq i\leq\iota\}
\end{equation}
with $\sigma_1=1$, satisfies $\|\delta_rx\|=r\,\|x\|$ and $\|xy\|\leq\|x\|+\|y\|$. 
We have denoted by $|\cdot|$ the underlying norm on $\vM$ that makes it a Banach space.
This convention will be understood in the sequel.
The properties of $\|\cdot\|$ that we have previously seen, allow us to say that 
$\|\cdot\|$ is a {\em Banach homogeneous norm} of $\vM$. 

If we set $\rho(x,y)=\|x^{-1}y\|$, then we have obtained a left invariant homogeneous distance on $\vM$ with respect to the group operation such that $\rho(\delta_rx,\delta_ry)=r\,\rho(x,y)$ for all
$x,y\in\vM$ and $r>0$. We say that $\rho$ is a {\em Banach homogeneous distance} on $\vM$.
In the sequel, we assume that every Banach homogeneous group $\vM$ is equipped with the
Banach homogeneous norm \eqref{Banhom} and the corresponding homogeneous distance $\rho$, unless otherwise stated.

For the subsequent examples, we recall the standard class of Banach spaces
\[
\ell^p=\bigg\{(x_k)_{k\geq0}\in\R^\N\,:\,\,\sum_{k=0}^\infty |x_k|^p<\infty\bigg\}
\]
where $p\geq1$ is any real number and $|(x_k)_{k\geq0}|_p=(\sum_{k=0}^\infty |x_k|^p)^{1/p}$.

\subsection{Two steps Banach homogeneous groups}

It is not difficult to construct the general model for a two step Banach homogeneous group.
We consider two Banach spaces $\vX$ and $\vT$. We have the Banach space $\vG_2=\vX\oplus\vT$ 
with the product norm. The structure of Banach homogeneous group is given by the 
bounded skew-symmetric bilinear form $\beta \colon \vX \times \vX \to \vT$ via the formula
$[(x,t),(x',t')]=\big(0,\beta(x,x')\big)$ for all $x,x'\in\vX$ and $t,t'\in\vT$.
Thus, the Lie group operation on $\vG_2$ is given by the following formula
\[
\quad (x,t)\cdot(x',t') = (x+x',t+t')+\big[(x,t),(x',t')\big]=(x+x',t+t'+\beta(x,x')).
\]
Let $|(x,t)|=|x|_{\vX}+|y|_{\vT}$ denote the product norm in the Banach space $\vG_2$. 
Let $c>0$ be such that $|\beta(x,x')|_{\vT}\leq c\,|x|_{\vX}|x'|_{\vX}$ for all $x,x'\in \vX$
and fix any constant $\sigma>0$ such that $\sigma\leq \sqrt{2/c}$.
Then the function $\|(x,t)\|=\max\{|x|_{\vX},\sigma|t|_{\vT}^{1/2}\}$ defines a homogeneous norm on $\vG_2$
and clearly for any $r>0$ the group isomorphism $\delta_r(x,t)=(rx,r^2t)$ for $(x,t)\in\vG_2$
is a dilation of $\vG_2$. In sum, only the mapping $\beta$ suffices to equip $\vG_2$ with the
structure of two step Banach homogeneous group.
%
%
%
%
\subsection{An infinite product of Heisenberg groups}\label{infHeis}
We wish to consider a concrete example of nontrivial two step Banach homogeneous group.
This group, that we denote by $\vH^\infty$, can be seen as a suitably topologized 
{\em infinite product of the same Heisenberg group}. 
As a Banach space $\vH^\infty$ coincides with $(\ell^2)^2\times\ell^1$,
where the horizontal subspace is $\vX =(\ell^2)^2$ and $\vT = \ell^1$. 
Any element $x\in\vH^\infty$ correspond to $(x_1,x_2,x_3)$ where $x_i=(x_{ij})_{j\geq 0}$. We also write 
$|x_i|_2=\sqrt{\sum_{j=0}^\infty x_{ij}^2}$ for $i=1,2$ and $|x_3|_1=\sum_{j=0}^\infty |x_{3j}|$.
For any $x,y\in\vH^\infty$, we define the skew-symmetric bilinear mapping $\beta:(\ell^2)^2\times(\ell^2)^2\lra\ell^1$ as follows 
\[
 \beta\big((x_1,x_2),(y_1,y_2)\big) = \big(0,0,(x_{1j}\,y_{2j}-x_{2j}\,y_{1j})_{j\geq0}\big).
\]
It follows that for all $(x_1,x_2),(y_1,y_2)\in(\ell^2)^2$ we have
\[
|\beta\big((x_1,x_2),(y_1,y_2)\big)|_1\leq \big(|x_1|_2^2+|x_2|_2^2\big)^{1/2}
\big(|y_1|_2^2+|y_2|_2^2\big)^{1/2}.
\]
According to the general model of two step Banach homogeneous group, the function
$\|(x_1,x_2,x_3)\|=\max\left\{\sqrt{|x_1|_2^2+|x_2|_2^2}, \sqrt{|x_3|_1}\right\}$
defines a homogeneous norm on $\vH^\infty$.
%
%
%
%
\subsection{An infinite product of Lipschitz maps}\label{infLip}
Let us consider any sequence of Lipschitz mappings $f^k:X\lra\H$, where 
$(X,d)$ is a metric space and $\H$ is the first Heisenberg group 
equipped with the homogeneous norm 
$|(\xi_1,\xi_2,\xi_3)|_\H=\max\Big\{|(\xi_1,\xi_2)|,\sqrt{|\xi_3|}\Big\}$
and the group operation 
$(\xi_1,\xi_2,\xi_3)(\eta_1,\eta_2,\eta_3)=(\xi_1+\eta_1,\xi_2+\eta_2,\xi_3+\eta_3+\xi_1\eta_2
-\xi_2\eta_1)$. We have denoted by $|\cdot|$ both the Euclidean norm in $\R^2$ and in $\R$.
Up to left translations, we can assume that for some $x_0\in X$ we have 
\begin{equation}\label{vanishk}
f^k(x_0)=0 \quad \mbox{ for all }\quad k\in\N.
\end{equation}
Let us define $\Lip(f^k)=\sup_{x,y\in X,\,x\neq y}\{ |f^k(x)^{-1}f^k(y)|_\H/d(x,y)\}$, then
we set $L_k=\Lip(f^k)$ and select any sequence $(r_k)_{k\geq0}$ of positive numbers such that
\begin{equation}\label{l2sum}
C_0=\bigg(\sum_{k=0}^\infty r_k^2\, L_k\bigg)^{1/2}<+\infty\,.
\end{equation}
We wish to construct the infinite product of the mappings $g_k=\delta_{r_k}\circ f^k$, where $k\in\N$.
We expect that the new target is the infinite product $\vH^\infty=(\ell^2)^2\times\ell^1$, defined in Subsection~\ref{infHeis}.
Following the notations of this subsection, we set $f^k(x)=(f^k_{11}(x),f^k_{12}(x),f^k_2(x))\in\H$,
so that 
\[
g_k(x)=\big(r_kf^k_{11}(x),r_kf^k_{12}(x),r_k^2f^k_2(x)\big)\in\H \quad \mbox{for all}\quad k\in\N.
\]
Setting $g^k(x)=(g^k_{11}(x),g^k_{12}(x),g^k_2(x))\in\H$, we define $G_{1j}(x)=(g^k_{1j}(x))_{k\geq0}$ with $j=1,2$ and $G_2(x)=(g_2^k(x))_{k\geq0}$.
Clearly, $\Lip(g^k)=r_kL_k$, therefore condition \eqref{vanishk} yields
\[
\max\bigg\{|g^k_1(x)|,\sqrt{|g^k_2(x)|}\bigg\}=|g^k(x)|_\H\leq r_k\,L_k\,d(x,x_0)
\]
where $g_1(x)=(g^k_{11}(x),g^k_{12}(x))$. By \eqref{l2sum}, it follows that
$G_{11}(x),G_{12}(x)\in\ell^2$ and $G_2(x)\in\ell^1$.
As a consequence, we have that $\big(G_1(x),G_2(x)\big)\in\vH^\infty$
for all $x\in X$, where we have defined $G_1(x)=(G_{11}(x),G_{12}(x))$. 
We use both the norm $\|\cdot\|$ and the group operation introduced in Subsection~\ref{infHeis}
for the Banach homogeneous group $\vH^\infty$. With these notions, for the mapping 
$G:X\lra\vH^\infty$ defined as $G(x)=\big(G_1(x),G_2(x)\big)$ for $x\in X$, we have
\[
\mbox{\small  $\displaystyle \|G(x)^{-1}G(y)\|=\max\Big\{|\!-G_1(x)+G_1(y)|_2,
\sqrt{|\!-G_2(x)+G_2(y)-G_{11}(x)\cdot G_{12}(y)+G_{12}(x)\cdot G_{11}(x)|_1} \Big\}$ }  
\]
where we have used the product $z\cdot w=\sum_{k=0}^\infty z_jw_j e_j\in\ell^1$, 
where $z,w\in\ell^2$ and
$(e_k)_{k\geq0}$ is the canonical Schauder basis of $\ell^1$. 
The condition \eqref{l2sum} finally leads us to the following Lipschitz continuity
\[
\|G(x)^{-1}G(y)\|\leq C_0\,d(x,y)\quad\mbox{for all}\quad x,y\in X.
\]
%
%
%
%
\subsection{Infinite products of Engel groups}
Let us consider the Engel group $\E$ with graded decomposition
$S_1\oplus S_2\oplus S_3$ and {\em graded} basis $(e_{11},e_{12},e_3,e_4)$,
namely $(e_{11},e_{12})$, $(e_3)$ and $(e_4)$ are bases of $S_1$, $S_2$ and $S_3$, respectively.
The only nontrivial bracket relations of $\E$ as a Lie algebra are 
$\cL(e_{11},e_{12})=e_3$ and $\cL(e_{11},e_3)=e_4$, then $\cL:\E\times\E\lra\E$ defines
a Lie product on $\E$. We define $H_1=(\ell^2)^2$, $H_2=\ell^2$, $H_3=\ell^p$ for some $p\geq1$
and set $\vE^\infty=H_1\times H_2\times H_3$.
An element $x$ of $\vE^\infty$ can be written as $(x_1,x_2,x_3)$, where 
$x_1=(x_{11},x_{12})$, $x_{1i}=(x_{1i}^k)_{k\geq0}\in\ell^2$, $i=1,2$,
$x_2=(x_2^k)_{k\geq0}\in\ell^2$ and $x_3=(x_3^k)_{k\geq0}\in\ell^p$. 

For all $\xi,\eta\in\ell^2$, we set $\xi\cdot\eta=\sum_{k=0}^\infty \xi^k\eta^k e^k\in\ell^p$,
where $(e^k)_{k\geq0}$ is the canonical Schauder basis of $\ell^p$. A nice point in the construction
of $\vE^\infty$ is that we do not need to construct the group operation, but it suffices to 
construct a continuous Lie product. Thus, we set
\[
[x,y]=\big(\,0\,,\,x_{11}\cdot y_{12}-x_{12}\cdot y_{11}\,,\, x_{11}\cdot y_2-x_2\cdot y_{11}\, \big).
\]
We fix the product Banach norm $|x|=|x_{11}|_2+|x_{12}|_2+|x_2|_2+|x_3|_p$ and observe that 
\[
|[x,y]|\leq 4 |x|\,|y|
\]
showing the continuity of $[\cdot,\cdot]$ with respect to the Banach norm $|\cdot|$.
The Jacobi identity follows from the one of $\cL(\cdot,\cdot)$.
\begin{Rem}\rm
Notice that the arbitrary choice of $p$ emphasizes, as one could expect, that there are
infinitely many Banach topologies to construct the infinite product of a Carnot group, 
as we will see in the next subsection.
\end{Rem}
%
%
%
%
\subsection{Infinite products of Carnot groups}\label{infCarnot}
The previous cases suggest a general ``product construction'' for any graded group
$\G=S_1\oplus\cdots\oplus S_\upsilon$. Thus, we set 
$\vG^\infty=H_1\times\cdots\times H_\upsilon$, where 
$H_i=(\ell^{p_i})^{n_i}$ and $n_i=\dim S_i$ for all $i=1,\ldots,\upsilon$
and the real numbers $p_i\geq1$, whenever $1\leq i,j\leq\upsilon$ and $i+j\leq\upsilon$,
satisfy the following inequality
\begin{equation}\label{pij}
p_{i+j}\geq\frac{1}{2}\max\{p_i,p_j\}.
\end{equation}
For any $i=1,\ldots,\upsilon$ we set the basis $(e_{i1},\ldots,e_{in_i})$ of $S_i$,
hence $(e_{iu})_{\mbox{\scriptsize $\substack{1\leq i\leq\upsilon \\1\leq u\leq n_i} $}}$ 
is a basis of $\G$.
For an element $x$ of $\vG^\infty$ we will use the equivalent notation
$(x_1,\ldots,x_\upsilon)$, where $x_i=(x_{i1},\ldots,x_{in_i})$ and 
$x_{iu}=(x_{iu}^k)_{k\geq0}\in\ell^{p_i}$. We set the norms
\begin{equation}\label{norms}
|x_i|_{p_i}=\left(\sum_{1\leq u\leq n_i}\sum_{k\geq0} |x_{iu}^k|^{p_i}\right)^{1/p_i}\quad\mbox{and}
\quad |x|=\sum_{i=1}^\upsilon|x_i|_{p_i}.
\end{equation}
Notice that we can also write 
$|x_i|=\left(\sum_{1\leq u\leq n_i} (|x_{iu}|_{p_i})^{p_i}\right)^{1/p_i}$.
Using the previous notation, for any $x\in\vG^\infty$ we set
\[
x^k=\sum_{\mbox{\scriptsize $\substack{1\leq i\leq\upsilon \\1\leq u\leq n_i} $}} x^k_{ij}\, e_{iu}
\in\G.
\]
Following this definition, for any $x,y\in\vG^\infty$ and any $k\in\N$ we set
\[
\cL(x^k,y^k)=\sum_{\mbox{\scriptsize $\substack{1\leq i\leq\upsilon \\1\leq u\leq n_i} $}}
\cL_{iu}^k(x,y)\,e_{iu}.
\]
Taking into account that 
\[
\cL(e_{iu},e_{jv})
=\sum_{\mbox{\scriptsize $\substack{1\leq i,j\leq\upsilon,\; i+j\leq\upsilon,\,1\leq u\leq n_i,\\
1\leq v\leq n_j, 1\leq r\leq n_{i+j} } $}}
\beta^r_{iu,jv}\, e_{(i+j)r}
\]
where the coefficients $\beta^r_{iu,jv}$ determine the Lie algebra structure of $\G$, we have
the formula
\[
\cL^k_{ir}(x,y)=\sum_{\mbox{\scriptsize $\substack{1\leq a,b\leq\upsilon,\; a+b=i,\\
1\leq u\leq n_a,\, 1\leq v\leq n_b} $}}
\beta^r_{au,bv}\,x^k_{au}\,y^k_{bv}
\]
where $k\in\N$, $i=1,\ldots,\upsilon$ and $r=1,\ldots, n_i$.
As a consequence, we can define the elements
\[
\cL_{ir}(x,y)=\big(\cL^k_{ir}(x,y)\big)_{k\geq0} \quad \mbox{and}\quad 
\cL_i(x,y)=\big(\cL_{i1}(x,y),\ldots,\cL_{in_i}(x,y)\big).
\]
By elementary computations, one can check that there exist constants $C_{1i}>0$ such that
\[
|\cL^k_{ir}(x,y)|^{p_i}\leq C_{1i} \sum_{\mbox{\scriptsize $\substack{ 
1\leq a,b\leq\upsilon \\ a+b=i} $}} (|x^k_a|_{p_a})^{p_i}\,(|y^k_b|_{p_b})^{p_i}
\]
where $x^k_a=(x^k_{a1},\ldots,x^k_{an_a})$ and 
$(|x^k_a|_{p_a})^{p_a}=\sum_{r=1}^{n_a}|x^k_{ar}|^{p_a}$ for any $x\in\vG^\infty$
and any $a=1,\ldots,\upsilon$. Thus, we can consider the sum with respect to $k$ and $r$, 
getting constants $C_{2i}>0$ such that
\begin{equation}\label{comp1}
\sum_{1\leq r\leq n_i}\sum_{k=0}^\infty|\cL^k_{ir}(x,y)|^{p_i}\leq C_{2i}
\sum_{\mbox{\scriptsize $\substack{1\leq a,b\leq\upsilon \\ a+b=i} $}}
\sum_{k=0}^\infty (|x^k_a|_{p_a})^{p_i}\,(|y^k_b|_{p_b})^{p_i}\,.
\end{equation}
Finally, we observe that 
\[
\bigg(\sum_{k=0}^\infty (|x^k_a|_{p_a})^{p_i}\,(|y^k_b|_{p_b})^{p_i}\bigg)^{1/p_i}\leq
|\big(|x^k_a|_{p_a}\big)_{k\geq0}|_{2p_i}\,|\big(|y^k_b|_{p_b}\big)_{k\geq0}|_{2p_i}
\]
and the condition \eqref{pij} yields
\[
\sum_{k=0}^\infty (|x^k_a|_{p_a})^{p_i}\,(|y^k_b|_{p_b})^{p_i}\leq
\bigg(|\big(|x^k_a|_{p_a}\big)_{k\geq0}|_{p_a}\,|\big(|y^k_b|_{p_b}\big)_{k\geq0}|_{p_b}\bigg)^{p_i}.
\]
Taking into account \eqref{norms}, we have 
$|\big(|x^k_a|_{p_a}\big)_{k\geq0}|_{p_a}=|x_a|_{p_a}$ 
and $|\big(|y^k_b|_{p_b}\big)_{k\geq0}|_{p_b}=|y_b|_{p_b}$.
As a result, taking into account \eqref{comp1}, we get
\[
\big(|\cL_i(x,y)|_{p_i}\big)^{p_i}=\sum_{1\leq r\leq n_i}
\sum_{k=0}^\infty|\cL^k_{ir}(x,y)|^{p_i}\leq C_{2i}
\sum_{\mbox{\scriptsize $\substack{1\leq a,b\leq\upsilon \\ a+b=i} $}}
(|x_a|_{p_a})^{p_i}\,(|y_b|_{p_b})^{p_i}\leq
C_{2i}\upsilon^2 |x|^{p_i}\,|y|^{p_i}\,,
\]
that immediately implies that 
\[
[x,y]=\big(0,\cL_2(x,y),\ldots,\cL_\upsilon(x,y)\big)\in\vG^\infty 
\quad \mbox{and} \quad
|[x,y]|\leq \mbox{\footnotesize $\displaystyle\sum_{i=2}^\upsilon (C_{2_i}\upsilon^2)^{1/p_i}$} \;|x|\,|y|.
\]
Finally, the Jacobi identity for the product $[\cdot,\cdot]$ follows from the Jacobi 
identity of $\cL$.
\begin{Rem}\rm
It is clear that the previous ``product construction'' can be suitably generalized 
to the cases of different Carnot groups. The obvious case is taking the 
product of $\vG^\infty$ with a different Carnot group $\G_1$, but many other
similar possibilities can arise.
\end{Rem}
%
%
%
%
%
%
%
%
%
%
%
%

\section{Differentiability}

This section is devoted to the proof of Theorem~\ref{thm:Pansudiff}. 
We equip a Carnot group $\G$ with a continuous left invariant distance $d$ such that
$d(\delta_rx,\delta_ry)=rd(x,y)$ for all $x,y\in\G$ and $r>0$, namely, a {\em homogeneous distance}.
The set $B_{x,r}\subset\G$ denotes the open ball of center $x$ and radius $r$ with respect to $d$.
When the center $x$ of the open ball is the origin, namely the unit element of $\G$, we simply
write $B_r$. The same rule is used for closed balls $D_{x,r}$ of center $x$ and radius $r>0$.
%
%
The set of {\em density points} $D(A)$ of $A\subset\G$ is formed by all $x\in\G$ with
\[
 \lim_{r\to0^+}\frac{\cH_d^Q(A\cap B_{x,r})}{\cH_d^Q(B_{x,r})}=1.
\]
In the sequel, $\vM$ is a Banach homogeneous group with gradation $H_1\oplus\cdots\oplus H_\iota$
and equipped with homogeneous norm $\|\cdot\|$ given by \eqref{Banhom}. The Carnot group
$\G$ has the decomposition into the direct sum $S_1\oplus\cdots\oplus S_\upsilon$, where the layers
satisfy the condition $[S_1,S_j]=S_{j+1}$ for all $j=1,\ldots,\upsilon-1$ and 
$[S_1,S_\upsilon]=\{0\}$. 
\begin{Def}\label{hhom}\rm 
A {\em homogeneous homomorphism}, in short h-homomorphism, from $\G$ to $\vM$ is
a continuous Lie group homomorphism $L:\G\lra\vM$ such that 
$L(\delta^\G_rx)=\delta^{\vM}_rL(x)$ for all $x\in\G$ and $r>0$,
where $\delta^{\G}_r$ and $\delta^{\vM}_r$ are dilations in $\G$ and $\vM$, respectively.
\end{Def}
\begin{Def}\label{def:hdiff}\rm 
Let $A\subset\G$ and let $\vM$ be a Banach homogeneous group equipped with a Banach homogeneous 
distance $\rho$. We say that $f:A\lra \vM$ is {\em differentiable} at the density point $x\in A$ 
if there exists an h-homomorphism $L:\G\lra\vM$ such that
\[
\rho(f(x)^{-1}f(xz),L(z))=o\big(d(z,0)\big)
\]
as $z\in x^{-1}A$ and $d(z,0)\to0^+$. 
The mapping $L$ is the differential of $f$ at $x$, that is uniquely defined
and denoted by $Df(x)$.
\end{Def}
In the sequel, saying that $H_1$ {\em has the RNP} precisely means that 
the restriction of the Banach norm of $\vM$ onto $H_1$ turns this closed subspace into
a Banach space with the RNP.
\begin{The}\label{the:Hdiff}
Let $\vM$ be a Banach graded Lie group such that $H_1$ has the RNP.
Let $A \subset \R$ and let $\gamma\colon A\to\vM$ be Lipschitz mapping. 
Then $\gamma$ is almost everywhere differentiable. 
\end{The}
\begin{proof}
We can obviously assume that $A$ is closed, since the target is a complete metric space.
Our Lipschitz bound on $\|\delta_{\frac{1}{h}}\big(\gamma(t)^{-1}\gamma(t+h)\big)\|$
for all $t,t+h\in A$ implies that 
\[
\frac{|\gamma_j(t+h)-\gamma_j(t)+\sum_{m=2}^\iota P_{mj}\big(-\gamma(t),\gamma(t+h)\big)|}{h^j}
\]
is also bounded, where $P_{mj}=\pi_j\circ P_m$. 
It is not restrictive to assume that $h>0$ and $A$ is bounded.
Since $\vM$ is a Banach space and $\gamma$ is also Lipschitz with respect to the Banach norm, 
we can apply the \cite{JLS86}, to get Lipschitz extension $f$ defined on a bounded interval 
containing $A$. In the sequel, we denote by $f_j$ the mapping $\pi_j\circ f$ for all $j+1,\ldots,\iota$.
The Radon-Nikodym property of $H_1$ implies that $f_1$ is a.e. differentiable on the bounded interval
containing $A$.
For a.e. $t$ in the bounded interval, we have
 \begin{equation}\label{intlim}
  \frac1h\int_{t-h}^{t+h} |\dot f_1(s) - \dot f_1(t) |\; ds \to 0\quad\mbox{as}\quad h \to 0^+,
 \end{equation}
where the integral is understood to be the Bochner integral, see \cite[Chapter 5]{BenLin}
for the basic properties of the Bochner integral.
We can restrict our attention to all density points $t$ of $A$ that are also points of 
differentiability and such that $\dot f_1$ satisfies \eqref{intlim}.
Now, we fix any point $t$ of $A$ having these properties.

We consider the left translated $g(s)=f(t)^{-1}f(s)\in\vM$. We also set $g_j=\pi_j\circ g$ for all $j=1,\ldots,\iota$, observing that $g_1(s)=f_1(s)-f_1(t)\in H_1$. 
Both mappings $g$ and $g_1$ are Lipschitz continuous with respect to $d$ and $|\cdot|$, respectively, and have the same Lipschitz constants of $f$ and $f_1$, respectively.

If we fix $n \in \N\sm\{0\}$, then we have $h_{n,t} > 0$, depending on $n$ and $t$, 
such that for all $h \in (A-t) \cap \left]0,h_{n,t}\right[$ we have $\dist(A,t+\frac{i}{n}h)<\frac{i}{n^2}h$. 
Thus, there  exist points $t_i \in A$, for $i = 0, 1, \dots, n$, so that 
\begin{equation}\label{tiA}
  \left|t + \frac{i}{n}h - t_i\right| \le \frac{h}{n},
\end{equation}
 $t_0 = t$ and $t_n = t + h$. We write $g_j(t+h)=A_j-B_j$, where 
\begin{eqnarray}\label{AjBj}
\left\{\begin{array}{l}
A_j=\sum_{i=0}^{n-1}\Big(g_j(t_{i+1})-g_j(t_i)+\sum_{m=2}^\iota P_{mj}(-g(t_i),g(t_{i+1})\big)\Big) \\
B_j=\sum_{i=0}^{n-1}\sum_{m=2}^\iota P_{mj}(-g(t_i),g(t_{i+1})\big)
\end{array}\right.\,.
\end{eqnarray}
We observe that 
\[
|A_j|\leq \sum_{i=0}^{n-1}|\pi_j\big(g(t_i)^{-1}g(t_{i+1})\big)|=
 \sum_{i=0}^{n-1}|\pi_j\big(f(t_i)^{-1}f(t_{i+1})\big)|,
\]
hence $A_j$ is bounded by $\sum_{i=0}^{n-1}|t_{i+1}-t_i|^j$, up to a constant factor only depending on the
Lipschitz constant of $f$. Since $|t_{i+1}-t_i|\leq 3h/n$, we get a constant $\kappa_{1j}>0$ such that
\begin{equation}\label{aj}
|A_j|\leq \kappa_{1j}\; h^j/ n^{j-1}.
\end{equation}
We denote by $l_0>0$ a number only depending on the Lipschitz constant of $f$ and such that 
\[
\max\{\|f(t')^{-1}f(t'')\|,|f(t')-f(t'')|,|f_1(t')-f_1(t'')|,\ldots,|f_\iota(t')-f_\iota(t'')|\}
\leq l_0\; |t'-t''|
\]
for all $t',t''$ belonging to the bounded interval containing $A$. 
By the Dynkin formula \eqref{dynkin}, the term $P_{mj}(-g(t_i),g(t_{i+1}))$ in \eqref{AjBj} can be written
as a linear combination of 
\begin{equation}
\pi_j\,\Big(\Delta_i\circ g(\tau_{1,i})\circ\cdots\circ g(\tau_{m-2,i})\Big)\,,
\end{equation}
where $\Delta_i\in\{-g(t_i)\circ g(t_{i+1}),g(t_{i+1})\circ\big(-g(t_i)\big)\}$ and $\tau_{1,i},\ldots,\tau_{m-2,i}\in\{t_i,t_{i+1}\}$. 
Therefore we have $\Delta_i\in\{\pm g(t_i)\circ g(t_{i+1})\}$
and the previous term can be written as
\begin{equation}\label{lcmb}
\pm \pi_j\,\Big(g(t_i)\circ \Big(g(t_{i+1})-g(t_i)\Big)\circ g(\tau_{1,i})\circ\cdots\circ g(\tau_{m-2,i})\Big)\,.
\end{equation}
Up to a change of sign, this term is the sum of elements 
\begin{equation}
\Big(g_{l_1}(t_{i+1})-g_{l_1}(t_i)\Big)\circ g_{l_2}(t_i)\circ g_{l_3}(\tau_{1,i})\circ\cdots\circ g_{l_m}(\tau_{m-2,i})\,,
\end{equation}
where $1\leq l_1,\ldots,l_m<j$ and $l_1+\cdots+l_m=j$.
Recall that the range of $m\in\N$ is $2\leq m\leq j$.
We start with the case $m=j$, where
\begin{equation}
\Big(g_1(t_{i+1})-g_1(t_i)\Big)\circ g_1(t_i)\circ g_1(\tau_{1,i})\circ\cdots\circ g_1(\tau_{m-2,i})\,.
\end{equation}
We can write this element as follows 
\[
\bigg(\int_{t_i}^{t_{i+1}}\dot g_1\bigg) \circ \bigg[\bigg(\int_t^{t_i}\dot g_1-\dot g_1(t)\bigg)
+(t_i-t)\dot g_1(t) \bigg]\circ g_1(\tau_{1,i})\cdots\circ g_1(\tau_{m-2,i}).
\]
This can be considered as the sum of 
$
\Big(\int_{t_i}^{t_{i+1}}\dot g_1\Big) \circ \Big(\int_t^{t_i}\dot g_1-\dot g_1(t)\Big)
\circ g_1(\tau_{1,i})\circ\cdots\circ g_1(\tau_{m-2,i})
$
and $\Big(\int_{t_i}^{t_{i+1}}\dot g_1\Big) \circ \Big((t_i-t)\dot g_1(t) \Big)\circ
g_1(\tau_{1,i})\circ\cdots\circ g_1(\tau_{m-2,i}).
$
The norm of the first element is not larger than
$\displaystyle 
l_0^{j-2}\,h^{j-2}\,\bigg(\int_{\min\{t_i,t_{i+1}\}}^{\max\{t_i,t_{i+1}\}}|\dot g_1|\bigg)
\, \bigg(\int_t^{t+h}|\dot g_1-\dot g_1(t)|\bigg)$,
hence we get  
{\small
\begin{equation}
\bigg|\Big(\int_{t_i}^{t_{i+1}}\dot g_1\Big) \circ \Big(\int_t^{t_i}\dot g_1-\dot g_1(t)\Big)
\circ \cdots\circ g_1(\tau_{m-2,i})\bigg|\leq l_0^{j-2} h^{j-2}\,\bigg(\int_{t+\frac{(i-1)}{n}h}^{t+\frac{(i+2)}{n}h}\big|\dot g_1\big|\bigg)
\bigg(\int_t^{t+h}|\dot g_1-\dot g_1(t)|\bigg).
\end{equation}
}
The second one can be written as 
\[
\bigg(\int_{t_i}^{t_{i+1}}\dot g_1-\dot g_1(t)\bigg) \circ \bigg((t_i-t)\dot g_1(t) \bigg)\circ
g_1(\tau_{1,i})\circ\cdots\circ g_1(\tau_{m-1,i})\,,
\]
hence its norm is less than or equal to 
$\displaystyle l_0^{j-1} h^{j-1}\,\int_{t+\frac{(i-1)}{n}h}^{t+\frac{(i+2)}{n}h}
\big|\dot g_1-\dot g_1(t)\big|$.
We have proved that
$|g_1(t_{i+1})\circ g_1(t_i)\circ g_1(\tau_{1,i})\circ\cdots\circ g_1(\tau_{m-2,i})|$ is less than
\[
 l_0^{j-2} h^{j-2}\,\bigg(\int_{t+\frac{(i-1)}{n}h}^{t+\frac{(i+2)}{n}h}\big|\dot g_1\big|\bigg)
\bigg(\int_t^{t+h}|\dot g_1-\dot g_1(t)|\bigg)
+l_0^{j-1} h^{j-1}\,\int_{t+\frac{(i-1)}{n}h}^{t+\frac{(i+2)}{n}h}\big|\dot g_1-\dot g_1(t)\big|.
\]
This proves that 
\begin{equation}\label{estim11}
 |g_1(t_{i+1})\circ g_1(t_i)\circ\cdots\circ g_1(\tau_{m-2,i})|\leq 
3\,l_0^{j-1}\,h^{j-1}\bigg(\frac{1}{n}\int_t^{t+h}|\dot g_1-\dot g_1(t)| +
\int_{t+\frac{(i-1)}{n}h}^{t+\frac{(i+2)}{n}h}\big|\dot g_1-\dot g_1(t)\big|\bigg)\,.
\end{equation}
Notice that in the case $j=2$, the estimate \eqref{estim11} can be read as 
\begin{equation}\label{estim112}
\Big|\sum_{m=2}^2 P_{m2}(-g_1(t_i),g_1(t_{i+1}))\Big|\leq 
\frac{3}{2}\,l_0\,h\;\bigg(\frac{1}{n}\int_t^{t+h}|\dot g_1-\dot g_1(t)| +
\int_{t+\frac{(i-1)}{n}h}^{t+\frac{(i+2)}{n}h}\big|\dot g_1-\dot g_1(t)\big|\bigg)\,,
\end{equation}
for all $0<h<h_{n,t}$, $t+h\in A$ and all possible choices of $t_i\in A$ satisfying \eqref{tiA},
that clearly also depend on $h$.
Arguing by induction, suppose that 
\begin{equation}\label{estimj}
\Big|\sum_{m=2}^l P_{ml}(-g(t_i),g(t_{i+1}))\Big|\leq 
\kappa_{2l}\,h^{l-1}\;\bigg(\frac{h}{n^2}+\frac{1}{n}\int_{t-h}^{t+h}|\dot g_1-\dot g_1(t)| +
\int_{t+\frac{(i-1)}{n}h}^{t+\frac{(i+2)}{n}h}\big|\dot g_1-\dot g_1(t)\big|\bigg)
\end{equation}
holds for all $l=2,\ldots,j$, for all $t,t+h\in A$, $0<h<h_{n,t,j}$ and any choice of $t_i\in A$ satisfying \eqref{tiA} for all $i=1,\ldots,n-1$, $t_0=t$ and $t_n=t+h$. 
In view of \eqref{estim112}, the induction hypothesis is true for $j=2$.
Thus, for any $t\in A$, we have to find an $h_{n,t,j+1}>0$ such that
\begin{equation}\label{indj+1}
\Big|\sum_{m=2}^{j+1} P_{m(j+1)}(-g(t_i),g(t_{i+1}))\Big|\leq \kappa_{2(j+1)}\,h^j\;\bigg(\frac{h}{n^2}+\frac{1}{n}\int_{t-h}^{t+h}|\dot g_1-\dot g_1(t)| +
\int_{t+\frac{(i-1)}{n}h}^{t+\frac{(i+2)}{n}h}\big|\dot g_1-\dot g_1(t)\big|\bigg)
\end{equation}
for any choice of $t_i\in A$ satisfying \eqref{tiA} and all $0<h<h_{n,t,j+1}$ such that $t+h\in A$.
Since $g_l(t+h)=A_l+B_l$, we observe that our inductive assumption yields
\begin{equation}\label{indl}
|g_l(t+h)|\leq 3\,\kappa_{2l}\,h^{l-1}\,\int_{t-h}^{t+h}|\dot g_1-\dot g_1(t)|+
(\kappa_{1l}+\kappa_{2l})\;\frac{ h^l}{n}
\end{equation}
for any $t,t+h\in A$ with $0<h<h_{n,t,j}$ and all $l=2,\ldots,j$, where $j\geq 2$.
Arguing as in the previous steps, by the Dynkin formula, the single addends 
$P_{m(j+1)}(-g(t_i),g(t_{i+1}))$ with $2\leq m\leq j$ are finite sums of elements
\begin{equation}\label{termj+1}
\Big(g_{l_1}(t_{i+1})-g_{l_1}(t_i)\Big)\circ g_{l_2}(t_i)\circ g_{l_3}(\tau_{1,i})\circ\cdots\circ g_{l_m}(\tau_{m-2,i})\,,
\end{equation}
where $1\leq l_1,\ldots,l_m\leq j$, $l_1+\cdots+l_m=j+1$ and $2\leq m\leq j+1$.
If $m=j+1$, then the general validity of \eqref{estim11} in our case gives 
\begin{equation}
 |g_1(t_{i+1})\circ g_1(t_i)\circ\cdots\circ g_1(\tau_{m-2,i})|\leq 
3\,l_0^j\,h^j\;\bigg(\frac{1}{n}\int_{t-h}^{t+h}|\dot g_1-\dot g_1(t)| +
\int_{t+\frac{(i-1)}{n}h}^{t+\frac{(i+2)}{n}h}\big|\dot g_1-\dot g_1(t)\big|\bigg)\,.
\end{equation}
for all $0<h<h_{n,t}$ and all $t_i\in A$ satisfying \eqref{tiA}.
Let us now consider the case $2\leq m<j+1$, where we have to apply our inductive hypothesis. 
Concerning \eqref{termj+1} we have two main cases. 

The first one is when $l_1=1$. Thus, we consider the set $J_0$ 
all integers $j\in\{2,3,\ldots,m\}$ such that $l_j\geq2$. Since $m<j+1$, we have $J_0\neq\emptyset$.
Precisely, $J_0$ is made by $p$ distinct elements $\{j_1,\ldots,j_p\}$ with $1\leq p\leq m-1$. 
We also set $I_0=\{2,\dots,m\}\sm J_0$. If $p<m-1$, namely $I_0\neq\emptyset$, then
$I_0=\{i_1,\ldots,i_{m-p-1}\}$ and $l_i=1$ for all $i\in I_0$. 
By the consequence \eqref{indl} of the inductive assumption, since $0<t_{i+1}-t,t_i-t\leq h<h_{n,t,j}$,
it follows that the norm of \eqref{termj+1} is less than or equal to
\[
l_0\,|t_{i+1}-t_i|\;\Big(l_0h\Big)^{m-p-1}\prod _{l=l_{j_1},\ldots,l_{j_p}}
\bigg(3\,\kappa_{2l}\,h^{l-1}\,\int_{t-h}^{t+h}|\dot g_1-\dot g_1(t)|+(\kappa_{1l}+\kappa_{2l})\;\frac{ h^l}{n}\bigg)
\]
for all $0<h<h_{t,n,j}$, with $t+h\in A$. Thus, there exists $\kappa_3(m,p)>0$, only
depending on $l_0$, $\kappa_{1l}$ and $\kappa_{2l}$, for all $l=2,\ldots,j$, such that 
\begin{equation}\label{casel1}
\mbox{\footnotesize $\displaystyle \Big|\Big(g_{l_1}(t_{i+1})-g_{l_1}(t_i)\Big)\circ g_{l_2}(t_i)
\circ g_{l_3}(\tau_{1,i})\circ\cdots\circ g_{l_m}(\tau_{m-2,i})\Big|$}\leq
\kappa_3(m,p)\, h^j\;\bigg(\frac{1}{n}\int_{t-h}^{t+h}|\dot g_1-\dot g_1(t)|+\frac{h}{n^2}\bigg).
\end{equation}
Explicitly, we can choose 
$
\kappa_3(m,p)=9\,l_0^{m-p}\,(12 l_0+1)^{p-1}\,\Big(\max_{2\leq l\leq j}\; \kappa_{1l}+\kappa_{2l}\Big)^p\,.
$

The remaining case is $l_1\geq 2$. 
We observe that for the integers $l$ such that $2\leq l\leq\iota$, 
the consequence \eqref{indl} of the inductive assumption yields  
\begin{equation}\label{estind}
|g_l(t')|\leq (t'-t)^l (\kappa_{1l}+\kappa_{2l}+12\,\kappa_{2l}\,l_0)
\end{equation}
for all $t'\in A$ such that $0<t'-t<h_{n,t,j}$. In the case $l=1$, we have 
$|g_1(t')|\leq l_0\,(t'-t)$ for all $t'\in A\cap (t,+\infty)$.
Now, the general term \eqref{termj+1} can be written as the sum of
\begin{equation}\label{fprod}
\Big(g_{l_1}(t_{i+1})-g_{l_1}(t_i)+\sum_{m=2}^{l_1}P_{ml_1}\big(-g(t_i),g(t_{i+1})\big)\Big)\circ g_{l_2}(t_i)\circ g_{l_3}(\tau_{1,i})\circ\cdots\circ g_{l_m}(\tau_{m-2,i})
\end{equation}
and of
\begin{equation}\label{estim_l1}
\Big(\sum_{m=2}^{l_1}P_{ml_1}\big(-g(t_i),g(t_{i+1})\big)\Big)\circ g_{l_2}(t_i)\circ g_{l_3}(\tau_{1,i})\circ\cdots\circ g_{l_m}(\tau_{m-2,i}).
\end{equation}
Since the first factor of \eqref{fprod} is $\pi_{l_1}\Big(g(t_i)^{-1}g(t_{i+1})\Big)$, 
the norm of \eqref{fprod} is not greater than 
\begin{equation}\label{estfprod}
l_0^{l_1}\frac{|t_i-t_{i+1}|^{l_1}}{\sigma_{l_1}}(\kappa_{1l}+\kappa_{2l}+l_0+12\kappa_{2l}l_0)^{m-1}\;h^{l_2+\cdots+l_m}\,.
\end{equation}
Now, we observe that the norm of \eqref{estim_l1} is less than or equal to the following number
\begin{equation}\label{secadd}
\Big|\sum_{m=2}^{l_1}P_{ml_1}\big(-g(t_i),g(t_{i+1})\big)\Big|\, |g_{l_2}(t_i)|\,|g_{l_3}(\tau_{1,i})|\;
\cdots\; |g_{l_m}(\tau_{m-2,i})|\,.
\end{equation}
By the induction hypothesis \eqref{estimj}, the first factor of this product is less than or equal to 
\[
\kappa_{2l_1}\,h^{l_1-1}\;\bigg(\frac{h}{n^2}+\frac{1}{n}\int_{t-h}^{t+h}|\dot g_1-\dot g_1(t)| +
\int_{t+\frac{(i-1)}{n}h}^{t+\frac{(i+2)}{n}h}\big|\dot g_1-\dot g_1(t)\big|\bigg)\,,
\]
hence, as before, taking into account \eqref{indl}, the product \eqref{secadd} is not larger than
\begin{equation}\label{estestim_l1}
\kappa_{2l_1}\,h^{l_1-1}\;\bigg(\frac{h}{n^2}+\frac{1}{n}\int_{t-h}^{t+h}|\dot g_1-\dot g_1(t)| +
\int_{t+\frac{(i-1)}{n}h}^{t+\frac{(i+2)}{n}h}\big|\dot g_1-\dot g_1(t)\big|\bigg)\,
\kappa_4(m)\;h^{l_2+\cdots+l_m}\,,
\end{equation}
where $\kappa_4(m)=(\kappa_{1l}+\kappa_{2l}+l_0+12\kappa_{2l}l_0)^{m-1}$.
Taking into account the decomposition of \eqref{termj+1} into the sum of 
\eqref{fprod} and \eqref{estim_l1} and using the estimates \eqref{estfprod} 
and \eqref{estestim_l1}, we get a geometric constant $\kappa_5(m)$, only depending
on $l_0$ and all $\kappa_{1,l}$ and $\kappa_{2l}$ with $l=2,\ldots,j$, such that
\begin{equation}\label{estin-termj+1}
\mbox{\scriptsize $\displaystyle \frac{\Big|\Big(g_{l_1}(t_{i+1})-g_{l_1}(t_i)\Big)\circ g_{l_2}(t_i)\circ\cdots\circ g_{l_m}(\tau_{m-2,i})\Big|}{\kappa_5(m)\,h^j}\leq $}\;
\mbox{\small $\displaystyle \frac{h}{n^2}+\frac{1}{n}\int_{t-h}^{t+h}|\dot g_1-\dot g_1(t)| +
\int_{t+\frac{(i-1)}{n}h}^{t+\frac{(i+2)}{n}h}\big|\dot g_1-\dot g_1(t)\big|$ }
\end{equation}
in the case $l_1\geq2$. Joining both cases $l_1=1$ and $l_1\geq 2$, namely, joining
\eqref{casel1} with \eqref{estin-termj+1}, we get a new constant 
$\kappa_6(m)\geq\kappa_5(m)$ depending on the same constants of $\kappa_5(m)$, such that
\begin{equation}\label{fin-estin-termj+1}
\mbox{\scriptsize $\displaystyle \frac{\Big|\Big(g_{l_1}(t_{i+1})-g_{l_1}(t_i)\Big)\circ g_{l_2}(t_i)\circ\cdots\circ g_{l_m}(\tau_{m-2,i})\Big|}{\kappa_6(m)\,h^j}\leq $}\;
\mbox{\small $\displaystyle \frac{h}{n^2}+\frac{1}{n}\int_{t-h}^{t+h}|\dot g_1-\dot g_1(t)| +
\int_{t+\frac{(i-1)}{n}h}^{t+\frac{(i+2)}{n}h}\big|\dot g_1-\dot g_1(t)\big|$ }
\end{equation}
whenever $\,2\leq m\leq j+1$, $l_1,l_2,\ldots,l_m\geq 1$, $l_1+l_2+\cdots+l_m=j+1$, $t,t+h\in A$, $0<h<\min\{h_{n,t,j},h_{n,t}\}$ and $t_i\in A$ satisfy \eqref{tiA} 
for all $i=1,\ldots,n-1$, where $t_0=t$ and $t_n=t+h$. Thus, under the same conditions, since
$P_{m(j+1)}(-g(t_i),g(t_{i+1}))$ is a finite linear combination of elements \eqref{termj+1}, 
we also have
\[
|P_{m(j+1)}(-g(t_i),g(t_{i+1}))|\leq \kappa_7(m)\,h^j\,\bigg(
\frac{h}{n^2}+\frac{1}{n}\int_{t-h}^{t+h}|\dot g_1-\dot g_1(t)| +
\int_{t+\frac{(i-1)}{n}h}^{t+\frac{(i+2)}{n}h}\big|\dot g_1-\dot g_1(t)\big|\bigg)
\]
for a suitable constant $\kappa_7(m)>0$, depending on $\kappa_6(m)$.
This immediately leads us to \eqref{indj+1} and concludes our argument by induction.
As a consequence, taking into account that $g_j(t+h)=A_j+B_j$ and 
$B_j=\sum_{i=0}^{n-1}\sum_{m=2}^jP_{mj}(-g(t_i),g(t_{i+1}))$, we get
\begin{equation}
|g_j(t+h)|\leq 3\,\kappa_{2j}\,h^{j-1}\,\int_{t-2h}^{t+2h}|\dot g_1-\dot g_1(t)|+
(\kappa_{1j}+\kappa_{2j})\;\frac{ h^j}{n}
\end{equation}
for all $j=2,\ldots,\iota$ and all $0<h<h_{n,t,j}$, that immediately leads us to the conclusion.
\end{proof}

%
%
%
%

\medskip

\begin{proof}[Proof of Theorem~\ref{thm:Pansudiff}]
We can clearly assume that $A$ is closed, since $f$ is Lipschitz and the target is 
a complete metric space. It is also not restrictive to assume that $A$ is also bounded. 
Let $v_1, v_2, \dots, v_N \in S_1$ be a set of horizontal directions of $\G$
such that $|v_i|=1$ for all $i=1,\ldots,N$ and for some $T>0$ we have that
 \[
  V = \{\delta_{t_1}v_1\cdots\delta_{t_N}v_N ~:~ |t_i|<T\}  
 \]
 is a neighbourhood of the origin containing $B_1$. For all $x \in \G$ and 
 $v \in \G\sm\{0\}$ we write the one dimensional set of parameters for which we hit 
 the set $A$ as $A(x,v) = \{t\in\R ~|~ x\delta_tv \in A\}$.
 Take one of the directions $v_i$ and denote by $Z_i$ the one-dimensional subgroup spanned by it. Let $W_i$ be the complementary subgroup, so that $\G$ is the semidirect product of $Z_i$ and $W_i$. 
 Then by Theorem~\ref{the:Hdiff} for any $x \in W_i$ the limit
 \begin{equation}\label{eq:directional}
  \lim_{\substack{t\to 0\\t \in A(x\delta_sv_i,v_i)}}
        \delta_{\frac{1}{t}}\big(f(x\delta_sv_i)^{-1} f(x\delta_sv_i \delta_tv_i)\big)
 \end{equation}
 exists for $\cL^1$-almost every $s \in A(x,v_i)$. Considering $W_i$ as a vector space and using 
 the  Fubini's theorem, we get the existence of the limit in \eqref{eq:directional}
 for $\cL^n$-almost every $(x,s) \in W_i\times \R$ for which $x\delta_sv_i \in A$.
 
Since both $\cH^Q$ with respect to the homogeneous distance $d$ and $\cL^n$ with respect
to the understood Euclidean metric on $\G$ are Haar measures of $\G$, 
applying the previous argument to all directions $v_i$, for $\cH^Q$-almost all $x \in A$ all the limits
 \begin{equation}\label{eq:uniform}
  \partial_{v_i}f(x) = \lim_{\substack{t\to 0\\t \in A(x,v_i)}}
        \delta_{\frac{1}{t}}\big(f(x)^{-1} f(x \delta_tv_i)\big) 
 \end{equation}
 with $i=1,2,\dots, N$ exist. Let us fix any $\epsilon>0$. Since $A$ is bounded, in view of both 
Lusin and Severini-Egorov theorems, we have a compact set $C \subset A$, made of density points,
such that for all $i=1,2, \dots, N$ 
the limits $\partial_{v_i}f(x)$ exist at every point $x \in C$, the convergence 
\[
\delta_{\frac{1}{t}}\big(f(x)^{-1} f(x \delta_tv_i)\big)\to \der_{v_i}f(x)\quad\mbox{as}\quad t\to0
\quad\mbox{and}\quad t\in A(x,v_i)
\]
is uniform on $C$, the maps $x \mapsto \partial_{v_i}f(x)$ are continuous on $C$, the convergence of the densities
 \begin{equation}\label{eq:uniformdirdens}
  \frac{\cH^1(B(x,r)\cap A \cap \{x\delta_tv_i : t \in \R\})}{2r} \to 1
 \end{equation}
as $r \downarrow 0$ is uniform for $x \in C$ and finally $\cH^Q(A\setminus C)< \epsilon$.
Let us choose $x\in C$, hence we know that
 \begin{equation}\label{eq:Cdensity}
   \lim_{d(0,z) \to 0}\frac{d(C,xz)}{d(0,z)} = 0.
 \end{equation}
Now, we choose $u = \delta_{t_1}v_1\cdots\delta_{t_N}v_N$ with $|t_i|<T$ and $t\in(-1,1)$.
Thus, we are able to find a sequence $v_1^t, \dots, v_N^t \in \G$ so that 
\begin{equation}\label{eq:dens}
xv_1^t\cdots v_i^t \in C\quad\mbox{ and }\quad 
 d(v_i^t,\delta_{tt_i}v_i) = d(C,xv_1^t\cdots v_{i-1}^t \delta_{tt_i}v_i)
\end{equation}
for all $i = 1, \dots, N$. Such a sequence exists since $C$ is compact.
We will also use the elements $w_1^t, \dots, w_N^t \in \G$ such that
for all $i = 1, \dots, N$ we have $xv_1^t\cdots v_{i-1}^tw_i^t \in A$,
\[
w_i^t \in \{\delta_h v_i : h \in \R\}\quad\mbox{and}\quad
  d(w_i^t,\delta_{tt_i}v_i) = d(\{xv_1^t\cdots v_{i-1}^t \delta_{h}v_i : h \in \R\}\cap A,xv_1^t\cdots v_{i-1}^t \delta_{tt_i}v_i).
 \]
Such a sequence exists because $A$ is closed and \eqref{eq:uniformdirdens} uniformly holds. 
Moreover, the same uniform convergence of \eqref{eq:uniformdirdens} yields 
 \begin{equation}\label{eq:huniform}
  \frac{d(w_i^t,\delta_{tt_i}v_i)}{t} \to 0
 \end{equation}
uniformly with respect to $x$ that varies in $C$, $|t_i|<T$ and $i=1,\ldots,N$, as $t\to0$. 
Notice that we have not emphasized the dependence on $x$.
The different sets of projected elements $\{v_i^t\}$ and $\{w_i^t\}$ are illustrated in Figure \ref{fig:Pansu}.
\begin{figure}
   \psfrag{a1}{$x\delta_{tt_1}v_1$}
   \psfrag{b1}{$xv_1^t$}
   \psfrag{a2}{$x\delta_{tt_1}v_1\delta_{tt_2}v_2$}
   \psfrag{b2}{$xv_1^tv_2^t$}
   \psfrag{c2}{$xv_1^tw_2^t$}
   \psfrag{x}{$x$}
   \psfrag{A}{$A$}
   \psfrag{C}{$C$}
   \centering
   \includegraphics[width=0.8\textwidth]{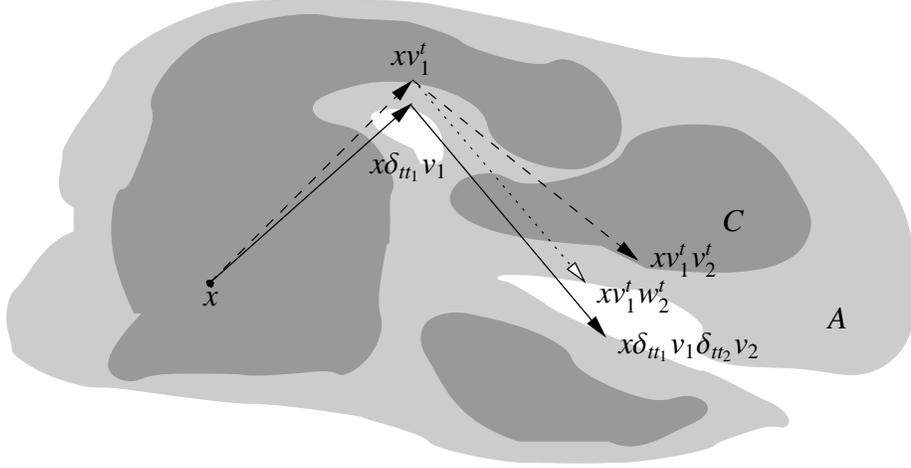}
   \caption{The choice of elements $v_i^t$ and $w_i^t$ in the proof of Theorem \ref{thm:Pansudiff}.}
   \label{fig:Pansu}
\end{figure}
For all $t\in(-1,1)$, we have 
\[
  \delta_{\frac{1}{t}}(f(x)^{-1}f(xv_1^t\cdots v_N^t)) = \prod_{i=1}^N\delta_{\frac{1}{t}}(f(xv_1^t\cdots v_{i-1}^t)^{-1}f(xv_1^t\cdots v_i^t)).
 \]
We now observe that for any $a>0$ at all points $z\in \tilde A$ where $\der_{v_i}f(z)$ exists, we also have 
the existence of $\der_{av_i}f(z)$ and there holds the homogeneity 
\begin{equation}\label{eq:hom}
\delta_a\der_{v_i}f(z)=\der_{av_i}f(z)\,.
\end{equation}
This allows us to consider all possible rescaled partial derivatives. Thus,
defining $\eta_{i}^t = xv_1^t\cdots v_{i-1}^t$ for any $i = 1, \dots, N$, we can consider
the inequality
 \begin{align*}
  \rho(\delta_{\frac{1}{t}}&(f(\eta_{i}^t)^{-1}f(\eta_{i}^tw_i^t)), \partial_{\delta_{t_i}v_i}f(x))
   \le \rho(\delta_{\frac{1}{t}}(f(\eta_{i}^t)^{-1}f(\eta_{i}^tw_i^t)), \partial_{\delta_{\frac1t}w_i^t}f(\eta_{i}^t)) \\
       & + \rho(\partial_{\delta_{\frac1t}w_i^t}f(\eta_{i}^t), \partial_{\delta_{t_i}v_i}f(\eta_{i}^t))
        + \rho(\partial_{\delta_{t_i}v_i}f(\eta_{i}^t), \partial_{\delta_{t_i}v_i}f(x))\,.
\end{align*}
The first addend on the right hand side uniformly converges to $0$ when $x$ varies in $C$,
as $t\to0$. In fact, $\eta_{i}^t \in C$ and the uniform convergence of \eqref{eq:uniform} is 
preserved under rescaling, where $\delta_{\frac1t}w_i^t = \delta_{h_tt_i}v_i$ and
$h_t$ uniformly converges to one on $C$ as $t\to0$, due to \eqref{eq:huniform}.
  
Since we have the rescaling of partial derivatives, we get
\[
\rho(\partial_{\delta_{\frac1t}w_i^t}f(\eta_{i}^t), \partial_{\delta_{t_i}v_i}f(\eta_{i}^t))=
\rho(\delta_{h_tt_i}\partial_{v_i}f(\eta_{i}^t), \delta_{t_i}\partial_{v_i}f(\eta_{i}^t))
\]
that uniformly goes to zero as $x\in C$ and $t\to0$, again due to \eqref{eq:huniform}.
This implies the uniform convergence of the second addend to zero.
Finally, the third term converges uniformly to $0$ by the continuity of $\partial_{v_i}f$ on $C$.
Joining \eqref{eq:Cdensity} and \eqref{eq:dens}, we have that
\begin{equation}\label{eq:uniformlyclose}
\delta_{1/t}(v_1^t\cdots v_i^t)\to\prod_{l=1}^i \delta_{t_l}v_l
\end{equation} 
uniformly with respect to all $i=1,\ldots,N$ and $|t_i|<T$ as $t \to 0$. It follows that 
\[
t^{-1}d(v_i^t,\delta_{tt_i}v_i)\to 0\quad\mbox{ as }\quad t\to0
\]
for all $i=1,\ldots,N$ and $|t_i|<T$.
This convergence is not uniform with respect to $x$, although we could make it
even uniform with respect to $x$ by choosing this element in a ``slightly smaller subset of $C$''. 
We have 
$t^{-1}d(w^t_i,v_i^t)\leq t^{-1}d(w_i^t,\delta_{tt_i}v_i)+t^{-1}d(\delta_{tt_i}v_i,v_i^t)\to0$
as $t\to0$. This gives 
 \[
  \delta_{\frac{1}{t}}(f(\eta_{i}^tw_i^t)^{-1}f(\eta_{i}^tv_i^t))  \to 0\quad\mbox{as}\quad
  t \to 0\,.
 \]
Combining the previous two limits we get that 
\[
  \delta_{\frac{1}{t}}(f(\eta_{i}^t)^{-1}f(\eta_{i}^tv_i^t)) 
   = \delta_{\frac{1}{t}}(f(\eta_{i}^t)^{-1}f(\eta_{i}^tw_i^t)) 
     \delta_{\frac{1}{t}}(f(\eta_{i}^tw_i^t)^{-1}f(\eta_{i}^tv_i^t))  \to \partial_{\delta_{t_i}v_i}f(x)
 \]
as $t \to 0$. This limit is uniform with respect to all $|t_i|<T$, but it may depend on $x$.
Therefore, we are lead to the existence of the following limit
 \begin{equation}\label{eq:diff}
  \lim_{t\to 0}\delta_{\frac{1}{t}}\big(f(x)^{-1} f(x v_1^t\cdots v_N^t)\big)
  = \prod_{i=1}^N \partial_{\delta_{t_i}v_i}f(x),
 \end{equation}
that is uniform with respect to $u \in V$.
This allows us to define $L_x(u) = \prod_{i=1}^N \partial_{\delta_{t_i}v_i}f(x)$ for all $u = \delta_{t_1}v_1 \cdots \delta_{t_N}v_N \in \G$.
Therefore the definition of $L_x(u)$ is independent of the choice of the representation of $u$. 
The choice of the ``nonlinear difference quotient'' 
\[
\delta_{\frac{1}{t}}\big(f(x)^{-1} f(x v_1^t\cdots v_N^t)\big)
\]
has been made in order to get the existence of the h-homomorphism $L_x$.
In fact, $L_x$ is an h-homomorphism since for all $u = \delta_{t_1}v_1 \cdots \delta_{t_N}v_N$ and
 $w = \delta_{\tau_1}v_1 \cdots \delta_{\tau_N}v_N$, we have
 \[
  L_x(uw) = \prod_{i=1}^N \partial_{\delta_{t_i}v_i}f(x) \prod_{i=1}^N \partial_{\delta_{\tau_i}v_i}f(x) = L_x(u)L_x(w).
 \]
The homogeneity of $L_x$ is obvious due to the homogeneity of partial derivatives \eqref{eq:hom}.
Let $(y_i)_{i=1}^\infty \subset A$ be any sequence so that $y_i \to x$. Define $\lambda_p = d(x,y_p)$ and let $z_p \in \G$ be such that $x^{-1}y_p = \delta_{\lambda_p}z_p$. It follows that
 \begin{align*}
  \lim_{p \to \infty} \frac{\rho(f(x)^{-1}f(y_p),L_x(x^{-1}y_p))}{d(x,y_p)} 
    & = \lim_{p \to \infty} \frac1{\lambda_p}\rho(f(x)^{-1}f(x\delta_{\lambda_p}z_p),L_x(\delta_{\lambda_p}z_p))\\
    & = \lim_{p \to \infty} \rho(\delta_{\frac1{\lambda_p}}(f(x)^{-1}f(x\delta_{\lambda_p}z_p)),L_x(z_p)) = 0
 \end{align*}
since both the limits \eqref{eq:diff} and \eqref{eq:uniformlyclose} for $i=N$ are uniform. 
Thus, $L_x$ is the h-differential of $f$ at $x$. The arbitrary choice of $\ep>0$ concludes the proof. 
\end{proof}

%
%
%
%

\section{The metric area formula}\label{Sec:metarea}


Let us fix a metric space $Y$ equipped with a distance $\rho$ and let $\G$ denote a stratified 
group equipped with both a homogeneous distance $d$ and the Hausdorff measure $\cH_d^Q$ constructed with respect to $d$.
The integer $Q$ is the homogeneous dimension of $\G$. Since $\G$ is a locally compact real Lie group, the measure $\cH_d^Q$ is 
the Haar measure of $\G$.
For all homogeneous distances $\sigma$ on $\G$, we set $\cH_\sigma^Q=\beta_Q\,h_\sigma^Q$ and 
\[
h_\sigma^Q(A)=\sup_{\ep>0}\;\inf\left\{\sum_{j=0}^\infty\frac{\diam_\sigma(E_j)^Q}{2^Q}: A\subset\bigcup_{j=0}^\infty E_j,\;
\diam_\sigma(E_j)\leq\ep\right\}
\]
where the constant $\beta_Q>0$ is fixed.
%
%
%
%
%
%
%
%
\begin{Def}[Metric differentiability]\label{def:metdif}\rm
Let $A\subset\G$ and let $f:A\lra Y$. We fix a {\em density point} $x\in A$.
Then we say that $f$ is {\em metrically differentiable} at $x$ if there exists
a homogeneous seminorm $s$ such that $\rho(f(x),f(xz))-s(z)=o(\|z\|)$ as 
$z\in x^{-1}A$ and $\|z\|\to0^+$.
The homogeneous seminorm $s$ is unique and it is denoted by $mdf(x)$,
which we call the {\em metric differential} of $f$ at $x$.
\end{Def}
Notice that in the case $\G$ is a Euclidean space, formula \eqref{gareametric} yields the 
Kirchheim's area formula established in \cite{Kir94}, where the following notion of Jacobian is used
\begin{equation}\label{KirJ}
 \cJ(s)=\frac{n\,\omega_n}{\int_{\S^{n-1}}s(x)^{-n}\,d\cH^{n-1}_{|\cdot|}(x)}\,.
\end{equation}
Here $s$ denotes a seminorm on $\R^n$ and $|\cdot|$ is the standard Euclidean norm.
In fact, the metric differential in this case is precisely a seminorm.
By the special geometric properties of normed spaces, Lemma~6 of \cite{Kir94} shows
that $\cJ(\|\cdot\|)=\cH^n_{\|\cdot\|}(A)/\cH^n_{|\cdot|}(A)$ for any $A\subset X$ of positive
measure, where $X$ is a finite dimensional Banach space with norm $\|\cdot\|$.
This immediately shows that formula \eqref{metJac} gives the Kirchheim's Jacobian when the Carnot group
is replaced by a Euclidean space.
The following remark shows that the metric Jacobian in stratified groups 
coincides with the sub-Riemannian Jacobian of \cite{Mag}.
\begin{Rem}\label{rem:BanHomJ}\rm
Let $L\colon \G\lra\vM$ be an injective h-homomorphism from a stratified group $\G$ to a Banach homogeneous 
group $\vM$. Let us define $s_L(x)=\rho(L(x),0)$ for all $x\in\G$ and notice that $s_L$ is a 
homogenous distance on $\G$, due to the injectivity of $L$. One can easily check that 
$\cH_{s_L}^Q(B_1)=\cH^Q_\rho(L(B_1))$. It follows that 
\begin{equation}\label{Js}
 J(s_L)=\frac{\cH^Q_{\rho}\big(L(B_1)\big)}{\cH_d^Q(B_1)}\,.
\end{equation}
When $\vM$ is in particular a Carnot group, then \eqref{Js} shows that
the metric Jacobian \eqref{metJac} coincides with the ``sub-Riemannian Jacobian'' 
introduced in Definition~10 of \cite{Mag}. 
One can relate these Jacobians with the classical ones computed by matrices. 
In fact, as proved in Proposition~3.18 of \cite{Mag}, there is a geometric constant $C$, 
a priori also depending on the subspace
$L(\G)$, such that 
\begin{equation}\label{Jacexpl}
J(s_L)=C\,\sqrt{\det(L_0^TL_0)}\,,
\end{equation}
where $L_0$ is the matrix representing $L$ as a linear mapping from $\G$ to $L(\G)$
with respect to a fixed scalar product. Of course, when one chooses distances with
many symmetries the constant $C$ will not depend on the subspace $L(\G)$.
Formula \eqref{Jacexpl} extends to our framework, since in the general case
where the target is a Banach homogeneous group $\vM$, we have that $L(\G)$ is still 
a finite dimensional linear subspace of $\vM$.
\end{Rem}
The key ingredient for the area formula is the linearization procedure associated to the
a.e. differentiability, see Lemma~3.2.2 of \cite{Fed}. For metric valued Lipschitz mappings 
with Euclidean source space, Kirchheim uses the separability of all compact convex sets,
according to 2.10.21 of \cite{Fed}, in order to get the following separability of norms:
there exists a countable family of norms $\{\|\cdot\|_i\}$ such that for every $0<\ep<1$
and every norm $\|\cdot\|$ we have some $\|\cdot\|_{i_0}$ such that 
$(1-\ep) \|\cdot\|_{i_0}\leq\|\cdot\|\leq (1+\ep)\|\cdot\|_{i_0}$.

Let us point out that a metric ball with respect to a homogenous distance need not be
a convex set. However, the previous separability of norms still holds for homogeneous norms,
since the point is that a Carnot group $\G$ is a boundedly compact metric space
and the class of nonempty compact sets in $\G$ is separable with respect to the 
Hausdorff distance between compact sets.
\begin{Lem}\label{lma:sepa}
There exists a countable family $\cF$ of homogeneous norms such that 
for every $\ep\in(0,1)$ and every homogeneous norm $\nu$, there exists
$s\in\cF$ such that 
\begin{equation}
(1-\ep)\,s\leq \nu\leq (1+\ep)\,s .
\end{equation}
\end{Lem}
Next, we recall the measure theoretic notion of Jacobian, see \cite{Mag22}.
\begin{Def}[Measure theoretic Jacobian]
Let $E\subset\G$ be a closed set, $f\colon E\lra Y$ Lipschitz and $x\in E$.
Then we define 
\[
 J_f(x)=\limsup_{r\to0^+}\frac{\cH^Q_\rho\big(f(E\cap D_{x,r})\big)}{\cH^Q_d(D_{x,r})}.
\]
\end{Def}
\begin{Lem}\label{lem:decJ}
Let $E\subset\G$ be closed and let $f\colon E\lra Y$ be Lipschitz.
Denote by $\cD\subset E$ the subset of points where $f$ is metrically differentiable and
the metric differential is a homogeneous norm. Then the following statements hold.
\begin{enumerate}
\item 
There exists a family of Borel sets $\{E_i\}_{i\in\N}$ such that
$\cD=\bigcup_{i=0}^\infty E_i$ and $f_{|E_i}$ is bi-Lipschitz onto its image.
\item
For $\cH^Q$-a.e. $x\in\cD$ we have $J\big(mdf(x)\big)=J_f(x)$.
\end{enumerate}
\end{Lem}
\begin{proof}
Let $\cF=\{s_i:i\in\N\}$ be as in Lemma~\ref{lma:sepa}.
Fix an arbitrary $\ep\in(0,1)$ and define for all $i,n\in\N$ the set
\[
 \cD_{i,n}=\{x\in\cD~:~ |\rho\big(f(xu),f(x)\big)-\,s_i(u)|\leq\ep s_i(u)
\text{ for all } u\in x^{-1}E \text{ with } \|u\|<e^{-n}\}\,.
\]
Combining metric differentiability with Lemma~\ref{lma:sepa}
gives $\cD=\bigcup_{i,n\in\N}\cD_{i,n}$.
Since $\G$ is separable, we can cover $\cD_{i,n}$ with a countable family of sets
$D_{z_l,e^{-n}/4}\cap\cD_{i,n}$ with $l\in\N$.
Then for all $x,y\in D_{z_l,e^{-n}/4}\cap\cD_{i,n}$, we immediately get
\begin{equation}\label{blep}
(1-\ep)\,s_i(x^{-1}y)\leq \rho\big(f(x),f(y)\big)\leq(1+\ep)\,s_i(x^{-1}y).
\end{equation}
This concludes the proof of the first statement.
Now, we fix $z\in\cD_{i,n}\cap D(\cD_{i,n})$. Then for all $r\in \left]0,e^{-n}/2\right[$ 
and $x,y\in D_{z,r}\cap \cD_{i,n}$, the inequalities \eqref{blep} again hold.
As a consequence, setting $L=\Lip(f)$, we get
\[
 \frac{\cH^Q_\rho\big(f(D_{z,r}\cap E)\big)}{\cH_d^Q(D_{z,r})}\leq
(1+\ep)^Q\frac{\cH^Q_{s_i}(D_{z,r})}{\cH_d^Q(D_{z,r})}+L^Q\frac{\cH_d^Q(D_{z,r}\sm\cD_{i,n})}{\cH_d^Q(D_{z,r})}
\]
and
\[
 \frac{\cH^Q_\rho\big(f(D_{z,r}\cap E)\big)}{\cH_d^Q(D_{z,r})}\geq
\frac{\cH^Q_\rho\big(f(D_{z,r}\cap\cD_{i,n})\big)}{\cH_d^Q(D_{z,r})}\geq
(1-\ep)^Q\frac{\cH^Q_{s_i}(D_{z,r}\cap\cD_{i,n})}{\cH_d^Q(D_{z,r})}.
\]
Therefore we have 
\[
(1-\ep)^Q\,\frac{\cH^Q_{s_i}(D_1)}{\cH_d^Q(D_1)}\leq\limsup_{r\to0^+} \frac{\cH^Q_\rho\big(f(D_{z,r}\cap E)\big)}{\cH_d^Q(D_{z,r})}\leq
(1+\ep)^Q\,\frac{\cH^Q_{s_i}(D_1)}{\cH^Q_d(D_1)}\,,
\]
for an arbitrary $\ep>0$. This leads us to the conclusion.
\end{proof}
\begin{Lem}\label{LipNegl}
Let $f:\G\lra Y$ be a Lipschitz mapping.
Let $E_0$ be the set of points $x\in\G$ for which there exists $v_x\in\G$
with $\|v_x\|=1$ such that for all $\ep>0$ we have $0<r_\ep<\ep$ 
such that 
\begin{equation}\label{mdsmall}
\rho\big(f(x\delta_rv_x),f(x)\big)\leq\ep\,r_\ep
\end{equation}
for all $0<r\leq r_\ep$. It follows that $\cH_\rho^Q\big(f(E_0)\big)=0$.
\end{Lem}
\begin{proof}
Let $L\geq1$ denote a Lipschitz constant of $f$, choose $x\in E_0$
and fix an arbitrary $R>0$. 
Take an arbitrary $\ep>0$ such that $4\ep<1$ and
$B_{x,\ep}\subset B_R$.
By definition of $E_0$, we have $v_x\in\S$ and $0<r_\ep<\ep$ such that
\eqref{mdsmall} holds. Defining 
\[
S_{v_x,\ep}=\bigcup_{0\leq r\leq r_\ep} B_{x\delta_rv_x,\frac{\ep r_\ep}{L}}\subset B_{x,\ep}\,,
\]
by triangle inequality, we get $f(S_{v_x,\ep})\subset B_\rho(f(x),2\ep r_\ep)$. 
Now, fix $N=[L/2\ep]$ and choose two distinct integers $i,j$ between $0$ and $N$.
Consider the elements 
\[
\xi\in B_{x\delta_{\frac{2\ep r_\ep i}{L}}v_x,\frac{\ep r_\ep}{L}}\quad\mbox{and}\quad
\eta\in B_{x\delta_{\frac{2\ep r_\ep j}{L}}v_x,\frac{\ep r_\ep}{L}},
\]
and notice that the previous balls are disjoint in view of the triangle inequality.
This leads us to $\cH_d^Q(S_{v_x,\ep})>C\,\ep^{Q-1}\,r_\ep^Q$.
As a result, the measure $\mu_R=(f_{|B_R})_\sharp\cH_\rho^Q$ on $Y$ satisfies
\[
\frac{\mu_R\Big(B_\rho(f(x),2\ep r_\ep)\Big)}{(2\ep r_\ep)^Q}\geq\frac{C}{2^Q\ep}\,.
\]
From the arbitrary choice of $\ep>0$, it follows that
\[
 \limsup_{t\to0^+}\frac{\mu_R\big(B_\rho(f(x),t)\big)}{t^{Q}}=+\infty
\]
for every $x\in E_0\cap B_R$, 
where $\mu_R$ is a finite Borel regular measure on $Y$.
Finally, standard differentiation theorems give $\cH_\rho^Q\big(f(E_0\cap B_R)\big)=0$
and the arbitrary choice of $R>0$ concludes the proof. 
\end{proof}
\begin{Cor}\label{cor:neg}
Let $E\subset\G$ be a closed set and let $f:E\lra Y$ be a Lipschitz mapping,
whose metric differential exists on a subset $E_0\subset E$ and at all point of
this set it is not a homogeneous norm. Then $\cH_\rho^Q\big(f(E_0)\big)=0$.
\end{Cor}
\begin{proof}
The image $f(E)$ is separable in $Y$, so in particular it is a
separable metric space that can be isometrically embedded into $l^\infty$.
Hence we can assume that $Y=l^\infty$.
With this target, the componentwise extension of $f$ immediately yields a Lipschitz extension $\tilde f$ 
defined on all of $\G$, having the same Lipschitz constant. 
Take $x\in E_0\subset D(E)$ and $v_x\in\G$ such that $\|v_x\|=1$ and $mdf(x)(v_x)=0$.
It is easy to check that $\tilde f$ is metrically differentiable at $x$ and 
$mdf(x)=md\tilde f(x)$.
Therefore $\tilde f$ satisfies conditions of Lemma~\ref{LipNegl} and our claim holds.
\end{proof}
\begin{proof}[Proof of Theorem~\ref{gmetarea}]
Since $\cH^Q_d$ is Borel regular and $f$ is Lipschitz, it is not
restrictive to assume that $A$ is closed and that $f$ is everywhere 
metrically differentiable.
By definition of metric Jacobian, $J(mdf(z))=0$ whenever $z$ belongs to the subset $A_0$ of $A$ 
where the metric differential is not a homogeneous norm. Corollary~\ref{cor:neg} implies 
$\cH^Q_\rho(f(A_0))=0$, hence \eqref{gareametric} holds for the restriction $f_{|A_0}$.
By Lemma~\ref{lem:decJ}, we have $A\sm A_0=\bigcup_{j=0}^\infty E_j$ where $E_j$ are Borel
sets and $f_{|E_j}$ is injective and $J_f$ equals $J(mdf(\cdot))$ $\cH^Q$-almost everywhere, hence Theorem~2 of \cite{Mag22} establishes \eqref{gareametric} for $f_{|(A\sm A_0)}$ and concludes the proof.
\end{proof}


\begin{thebibliography}{99}

\bibitem{BenLin}{\sc Y. Benyamini, J. Lindenstrauss},
{\em Geometric nonlinear functional analysis. {V}ol. 1},
American Mathematical Society Colloquium Publications, Vol. 48, (2000)

\bibitem{CheKle09}{\sc J. Cheeger, B. Kleiner},
{\em Differentiability of Lipschitz maps from metric measure spaces to Banach spaces with the Radon-Nikod\'ym property},
Geom. Funct. Anal. {\bf 19}, n.4, 1017-1028, (2009)

\bibitem{CheKle10a}{\sc J. Cheeger, B. Kleiner},
{\em Differentiating maps into $L^1$ and the geometry of BV functions},
Ann. of Math. (2) {\bf 171}, n. 2, 1347–1385, (2010)

\bibitem{CheKle10b}{\sc J. Cheeger, B. Kleiner},
{\em Metric differentiation, monotonicity and maps to $L^1$}, 
Invent. Math. {\bf 182}, n.2, 335–370 (2010)

\bibitem{Dyn47}{\sc E. B. Dynkin}, 
{\em Computation of the coefficients in the formula of Campbell-Hausdorff},
(Russian), Dokl. Akad. Nauk. SSSR (N.S.), {\bf 57 }, 323-326, (1947)

\bibitem{Dyn53}{\sc E. B. Dynkin}, 
{\em Normed Lie Algebras and Analytic Groups}, Amer. Math. Soc.
Transl., {\bf 97}, (1953)

\bibitem{Fed}{\sc H. Federer},
{\em Geometric Measure Theory},  Springer,  (1969)

\bibitem{FS82}{\sc G. B. Folland, E. M. Stein},
{\em Hardy Spaces on Homogeneous groups}, Princeton University Press, (1982)

\bibitem{Good76}{\sc R. W. Goodman},
{\em Nilpotent Lie groups: structure and applications to analysis}, 
Lecture Notes in Mathematics, {\bf 562}, Springer-Verlag, Berlin-New York, (1976)

\bibitem{Gr1}{\sc M. Gromov},
{\em Carnot-Carath\'eodory spaces seen from
within}, in {\em Subriemannian Geometry}, Progress in Mathematics,
{\bf 144}. ed. by A.Bellaiche and J.Risler, Birkhauser Verlag, Basel, (1996)

\bibitem{JLS86}{\sc W. B. Johnson, J. Lindenstrauss, G. Schechtman},
{\em Extensions of Lipschitz maps into Banach spaces}, Israel Journal of Math.,
{\bf 54}, n.2, 129-138, (1986)

\bibitem{Kir94} {\sc B. Kirchheim},
{\em Rectifiable metric spaces: local
structure and regularity of the Hausdorff measure},
Proc. Amer. Math. Soc., {\bf 121}, 113-123, (1994)

\bibitem{KirMag2}{\sc B. Kirchheim, V. Magnani},
{\em A counterexample to metric differentiability},
Proc. Ed. Math. Soc., {\bf 46}, 221-227, (2003)

\bibitem{Mag}{\sc V. Magnani},
{\em Differentiability and Area formula on stratified Lie groups},
Houston J. Math., {\bf 27}, n.2, 297-323, (2001)

\bibitem{Mag10}{\sc V. Magnani},
{\em Unrectifiability and rigidity in stratified groups},
Arch. Math., {\bf 83}, n.6, 568-576, (2004)

\bibitem{Mag22}{\sc V. Magnani},
{\em An area formula in metric spaces}, Coll. Math., to appear

\bibitem{Pan89}{\sc P. Pansu},
{\em M\'etriques de Carnot-Carath\'eodory quasiisom\'etries des
espaces sy\-m\'e\-tri\-ques de rang un}, Ann. Math., {\bf 129},
1-60, (1989)

\bibitem{Pau01}{\sc S. D. Pauls},
{\em The Large Scale Geometry of Nilpotent Lie groups},
Comm. Anal. Geom., {\bf 9}, n.5, 951-982, (2001)

\bibitem{Rog07}{\sc K. Rogovin},
{\em Non-smooth analysis in infinite dimensional Banach homogenous groups},
J. Convex Anal., {\bf 14}, n.4, 667-691, (2007)

\bibitem{Stein93}{\sc E. M. Stein},
{\em Harmonic Analysis: Real-Variable Methods, Orthogonality, and Oscillatory Integrals},
Princeton University Press, (1993)
\end{thebibliography}
\end{document}